\documentclass[12pt,a4paper,reqno]{amsart}
\usepackage{a4wide}
\usepackage[english,russian]{babel}
\usepackage[T2A]{fontenc}
\usepackage[cp1251]{inputenc} 
\usepackage{amsfonts}
\usepackage{amssymb, amsthm, amscd}
\usepackage{amsmath}
\usepackage{mathtools}
\usepackage{needspace}
\usepackage{etoolbox}
\usepackage{lipsum}
\usepackage{comment}
\usepackage{cmap}
\usepackage[pdftex]{graphicx}
\usepackage[unicode]{hyperref}
\usepackage[matrix,arrow,curve]{xy}
\usepackage[usenames,dvipsnames]{xcolor}
\usepackage{colortbl}
\usepackage{textcomp}
\usepackage{cite}
\usepackage{euscript}

\DeclareMathOperator{\Lie}{Lie}

\newcommand{\Z}{\mathbb{Z}}

\newcommand{\N}{\mathbb{N}}
\newcommand{\R}{\mathbb{R}}

\newcommand{\F}{\mathbb{F}}
\newcommand{\M}{\mathfrak{M}}

\newcommand{\ovl}{\overline}

\newcommand{\eps}{\varepsilon}
\newcommand{\ph}{\varphi}

\newcommand{\sub}{\subseteq}
\newcommand{\Hom}{\mathop{\mathrm{Hom}}\nolimits}

\newcommand{\lev}{\mathop{\mathrm{lev}}\nolimits}

\newcommand{\Max}{\mathop{\mathrm{Max}}\nolimits}
\newcommand{\rk}{\mathop{\mathrm{rk}}\nolimits}

\renewcommand{\ge}{\geqslant}
\renewcommand{\le}{\leqslant}
\newcommand{\sm}{\setminus}
\newcommand{\map}[3]{#1\colon #2\to #3}
\newcommand{\phan}{\phantom{,}}

\newcommand{\ox}{\otimes}

\newcommand{\<}{\langle}
\renewcommand{\>}{\rangle}

\newcommand{\emp}{\varnothing}

\renewcommand{\P}{\EuScript{P}}
\newcommand{\euQ}{\EuScript{Q}}
\newcommand{\euL}{\EuScript{L}}
\newcommand\leftact[2]{\phan^{#1}#2}

\theoremstyle{plain}
\newtheorem{thm}{Theorem}
\newtheorem{lem}{Lemma}
\newtheorem{prop}{Proposition}

\theoremstyle{definition}

\newtheorem{cor}{Corollary}

\theoremstyle{remark}
\newtheorem*{rem}{Remark}
\newtheorem*{warn}{Warning}

\AtBeginEnvironment{thm}{\begin{samepage}}
\AtEndEnvironment{thm}{\end{samepage}}
\AtBeginEnvironment{lem}{\begin{samepage}}
\AtEndEnvironment{lem}{\end{samepage}}
\AtBeginEnvironment{st}{\begin{samepage}}
\AtEndEnvironment{st}{\end{samepage}}
\AtBeginEnvironment{crit}{\begin{samepage}}
\AtEndEnvironment{crit}{\end{samepage}}
\AtBeginEnvironment{ax}{\begin{samepage}}
\AtEndEnvironment{ax}{\end{samepage}}
\AtBeginEnvironment{defn}{\begin{samepage}}
\AtEndEnvironment{defn}{\end{samepage}}
\AtBeginEnvironment{cor}{\begin{samepage}}
\AtEndEnvironment{cor}{\end{samepage}}
\AtBeginEnvironment{note}{\begin{samepage}}
\AtEndEnvironment{note}{\end{samepage}}
\AtBeginEnvironment{prop}{\begin{samepage}}
\AtEndEnvironment{prop}{\end{samepage}}

\newcommand{\Stab}{\mathop{\mathrm{Stab}}\nolimits}
\newcommand{\SL}{\mathop{\mathrm{SL}}\nolimits}

\newcommand{\Tran}{\mathop{\mathrm{Tran}}\nolimits}

\newcommand{\D}{\mathop{\mathrm{D}}\nolimits}

\renewcommand{\tilde}{\widetilde}
\renewcommand{\hat}{\widehat}

\newcommand{\vpi}{\varpi}

\DeclareMathOperator{\SC}{sc}
\DeclareMathOperator{\gen}{gen}
\DeclareMathOperator{\FG}{FG}

\makeatletter
\def\@settitle{\begin{center}%
    \baselineskip14\p@\relax
    \bfseries
    \@title
  \end{center}%
}

\def\@evenhead{\hfil\sc p. gvozdevsky\hfil}
\def\@oddhead{\hfil\sc overgroups of subsystem subgroups \hfil}
\makeatother

\title{Overgroups of subsystem subgroups in exceptional groups: $2A_1$-proof}

\address{Chebyshev Laboratory,\\ St. Petersburg State \\ University, 
	14th Line V.O., 29, \\ Saint Petersburg 199178 Russia.}
\thanks{Research is supported by Russian Science Foundation grant (project \textnumero\,17-11-01261).}

\author{Pavel Gvozdevsky}
\date{}

\begin{document}
\selectlanguage{english}

\maketitle
%\tableofcontents 

\begin{abstract}
	In the present paper, a weak form of sandwich classification for the overgroups of the subsystem subgroup $E(\Delta,R)$ of the Chevalley group $G(\Phi,R)$ is proved in case where $\Phi$ is a simply laced root system and $\Delta$ is its sufficiently large subsystem. Namely it is shown that for any such an overgroup $H$ there exists a unique net of ideals $\sigma$ of the ring $R$ such that  $E(\Phi,\Delta,R,\sigma)\le H\le \Stab_{G(\Phi,R)}(L(\sigma))$, where $E(\Phi,\Delta,R,\sigma)$  is an elementary subgroup associated with the net and $L(\sigma)$ s a corresponding subalgebra of the Chevalley Lie algebra.  
\end{abstract}

\section{Introduction}

In the paper \cite{Aschbacher84}, devoted to the {\it maximal subgroup project}, Michael Aschbacher introduced eight classes $C_1$--$C_8$ of subgroups of finite simple classical groups. The groups from these classes are "obvious" maximal subgroups of a finite classical groups. To be precise, each subgroup from an Aschbacher class either is maximal itself or is contained in a maximal subgroup that in its turn either also belongs to an Aschbacher class or can be constructed by a certain explicit procedure.

Nikolai Vavilov defined five classes of "large" subgroups of the Chevalley groups (including exceptional ones) over arbitrary rings (see, for example, \cite{StepDiss} for details). Although these subgroups are not maximal, he conjectured that they are sufficiently large for the corresponding overgroup lattice to admit a description. One of these classes is the class of subsystem subgroups (the definition will be given in Subsection 2.1).

In the present paper, we study the overgroups of subsystem subgroups in exceptional groups. The fundamental difference of this paper from all the previous ones is that here we use subsystems of type $2A_1$ to extract root elements. In all the papers published before, for this purpose they used ir\-reduci\-ble subsystems of rank at least 2.

To put the results of the present paper in a context, we now recall main results that are known at the moment.

\begin{itemize}	
	\item In \cite{BV84,KoibaevBlockdiag} and also in some other papers, the overgroups of (elementary) subsystem subgroups in general linear group were studied. In this case, subsystem subgroups are the groups of block-diagonal matrices.
	
	\item Further in the thesis of Nikolai Vavilov (see, in particular, \cite{VavMIAN} and \cite{VavSplitOrt}), this results were generalised for the cases of orthogonal and symplectic groups assuming that $2\in R^*$. The full proofs were published later. After that in the thesis of Alexander Shegolev \cite{SchegDiss} this assumption was lifted and also the similar problem for unitary groups was solved (see also \cite{SchegMainResults} and \cite{SchegSymplectic}), which closes the case of classical groups almost entirely.
	
	\item The cases of general linear and unitary groups can be generalised to certain noncommutative rings, but these ring should satisfy some other condition (for example, to be quasifinite or PI). The papers \cite{GolubchikSubgroups,StepOLD} and \cite{VavStepSubgroupsGLStability} were devoted to such generalisations.
	
	\item It turns out that to describe overgroups of subsystem subgroups in exceptional groups  over a commutative ring (see, for example, \cite[Pro\-blem~7]{VavStepSurvey}) is a much harder problem. The first step in this direction is a paper \cite{VSch}. The table from that paper contains the list of pairs $(\Phi,\Delta)$  for which such a description may be possible in principle, along with the number of ideals determining the level and along with certain links between these ideals.
	
	\item In the paper \cite{GvozLevi}, author solved this problem for the subsystems $A_{l-1}\le D_l$,  $D_5\le E_6$ and $E_6\le E_7$. For the simply laced systems, these are exactly the cases where the subsystem subgroup is a Levi subgroup, and the corresponding unipotant radical is abelian.
	
	\item Finally, note that the result of the paper \cite{LuzF4E6}, which describes overgroups of $F_4$ in $E_6$, is not a special case of our problem, but is pretty close to it.
\end{itemize}

Now we recall how does the answer in the results mentioned above usually look like. 

Let $G$ be an abstract group, and let $\euL$ be a certain lattice of its subgroups.  The lattice $\euL$ admits sandwich classification if it is a disjoint union of "sand\-wi\-ches":
\begin{align*}
	\euL&=\bigsqcup_i L(F_i,N_i),\\
	L(F_i,N_i)&=\{H\colon F_i\le H\le N_i\},
\end{align*}
where $i$ runs through some index set, and $F_i$ i
is a normal subgroup of $N_i$ for all $i$. To study such a lattice, it suffices to study the quotients $N_i/F_i$. In \cite{VSch} it was conjectured that
the lattice of subgroups of a Chevalley group that contain a sufficiently large subsystem subgroup admits sandwich classification for certain $F_i$ and $N_i$. Such theorems are also called {\it the standard description}. 

However, (at least) in the cases where the subsystem has an irreducible component of type $A_1$, the conjectures in \cite{VSch} should be reformulated. The reason why they cannot be true as stated is that the elementary subgroup of $\SL(2,R)$ is not normal in general.

Therefore, the main result of the present paper is similar to sandwich clas\-si\-fi\-ca\-tion, but, unlike all the other papers, the subgroup $F_i$ is not normal in $N_i$ in general.

The paper \cite{VavGav} is devoted to the $A_2$-proof of the structure theorems. That is the proof that employs the elements of the form
$$
x_\alpha(\xi)x_\beta(\zeta), \,\mbox{ where }\, \angle(\alpha,\beta)=\pi/3,
$$
to get into a parabolic subgroup. After the proof of the main lemma in \cite{VavGav} there is a remark, which says that to get into a parabolic subgroup one can also use the elements of the form $x_\alpha(\xi)x_\beta(\zeta)$, where $\angle(\alpha,\beta)=\pi/2$. Our method of proof is partially pased on this remark. Such method should be called the $2A_1$-proof, and it allows us to study the overgroups of subsystem subgroups even if the subsystem has type $nA_1$.

I am grateful to my teacher Nikolai Vavilov for setting the problem and for extremely
helpful suggestions.

\section{Basic notation}

\subsection{Root systems and Chevalley groups}

Let $\Phi$ be an irreducible root system in the sense of \cite{Bourbaki4-6}, $\P$ a lattice that is intermediate between the root lattice $\euQ(\Phi)$ and the weight lattice $\P(\Phi)$, $R$ a commutative associative ring with unity, $G(\Phi,R)=G_\P(\Phi,R)$ a Chevalley group of type $\Phi$ over $R$, $T(\Phi,R)=T_\P(\Phi,R)$ a split maximal torus of $G(\Phi,R)$. For every root $\alpha\in\Phi$ we denote by $X_\alpha=\{x_\alpha(\xi),\colon \xi\in R\}$ the corresponding unipotent root subgroup with respect to $T$. We denote by $E(\Phi,R)=E_\P(\Phi,R)$ the elementary subgroup generated by all $X_\alpha$, $\alpha\in\Phi$. 

In the rest of the papaer we always assume that $\Phi$ is simply laced.

Let $\Delta$ be a subsystem of $\Phi$. We denote by $E(\Delta,R)$ the subgroup of $G(\Phi,R)$, generated by all~$X_\alpha$, where $\alpha\in \Delta$. It is called an (elementary) {\it subsystem subgroup}. It can be shown that it is an elementary subgroup of a Chevalley group $G(\Delta,R)$, embedded into the group $G(\Phi,R)$. Here the lattice between $\euQ(\Delta)$ and $\P(\Delta)$ is an orthogonal projection of $\P$ onto the corresponding subspace.

We denote by $W(\Phi)$ resp. $W(\Delta)$ the Weyl groups of the systems $\Phi$ resp. $\Delta$.

We are going to describe intermediate subgroups between $E(\Delta,R)$ and $G(\Phi,R)$. The pair $(\Phi,\Delta)$ here should satisfy certain combinatorial condition, which we formulate in Subsection \ref{CombCondition}.

\subsection{Affine schemes}
The functor $G(\Phi,-)$ from the category of rings to the category of groups is an affine group scheme (a Chevalley--Demazure scheme). This means that its composition with the forgetful functor to the category of sets is representable, i.e.,
$$
G(\Phi,R)=\Hom (\Z[G],R).
$$
The ring $\Z[G]$ here is called the {\it ring of regular functions} on the scheme $G(\Phi,-)$. 

An element $g_{\gen}\in G(\Phi,\Z[G])$ that corresponds to the identity ring ho\-mo\-mo\-rphism is called the {\it generic element} of the scheme $G(\Phi,-)$. This element has a universal property: for any ring $R$ and for any $g\in G(\Phi,R)$, there exists a unique ring homomorphism
$$
\map{f}{\Z[G]}{R}
$$
such that $f_*(g_{\gen})=g$. For details about application of the method of generic elements
to the problems similar to that of ours, see the paper of Alexei Ste\-pa\-nov~\cite{StepUniloc}.

\subsection{Reduction homomorphism and congruence subgroups}

For a ring $R$ and an ideal $I\unlhd R$, we denote by $\rho_I$ the projection onto the quotient ring $R\to R/I$. The same notation stands for the reduction homomorphism $G(\Phi,R)\to G(\Phi,R/I)$ induced by this projection.

The kernel of this homomorphism is called the {\it principal congruence subgroup} of level $I$, and we denote it by $G(\Phi,R,I)$.

\subsection{Parabolic subgroups}

Let $\alpha_1,\alpha_2\in \Phi$ and $\alpha_1\perp\alpha_2$. We introduce a notation for the following linear functional on the span of $\Phi$:  
$$
\vpi_{\alpha_1,\alpha_2}(\gamma)=(\alpha_1+\alpha_2,\gamma).
$$
Since $\Phi$ is simply laced, it follows that for any $\gamma\in \Phi$ we have $\vpi_{\alpha_1,\alpha_2}(\gamma)\in\{-2,-1,0,1,2\}$. Note that the values $\pm 1$ may fail to occur. The set of all roots such that the value of $\vpi_{\alpha_1,\alpha_2}$ on them is nonnegative is a parabolic set. We denote by $P_{\alpha_1,\alpha_2}$ the corresponding parabolic subgroup. 

Further we introduce the following notation:
\begin{align*}
	&\Sigma_{\alpha_1,\alpha_2}=\{\gamma\in\Phi\colon\vpi_{\alpha_1,\alpha_2}(\gamma)=2\},\\
	&U'_{\alpha_1,\alpha_2}=\<x_{\gamma}(\xi)\colon\gamma\in\Sigma_{\alpha_1,\alpha_2},\, \xi\in R\>\le P_{\alpha_1,\alpha_2}.
\end{align*}
In other words, the group $U'_{\alpha_1,\alpha_2}$ is a unipotent radical of $P_{\alpha_1,\alpha_2}$ if the functional $\vpi_{\alpha_1,\alpha_2}$ gives 3-grading on our root system (that is the values $1$ and $-1$ do not occur on roots). If the functional $\vpi_{\alpha_1,\alpha_2}$ gives a non-degenerate 5-grading, then the group $U'_{\alpha_1,\alpha_2}$ is a derived subgroup of the unipotent radical.

\subsection{Group theoretic notation}
\begin{itemize}
	\item Recall that for abstract groups $A,B\le G$,  the transporter from $A$ to $B$ is the set
	$$
	\Tran_G(A,B)=\{g\in G\colon gAg^{-1}\le B\}.
	$$
	
	\item If the group $G$ acts on the set  $X$, $x\in X$ and $Y\sub X$, we denote by $\Stab_G(x)$ resp. $\Stab_G(Y)$ the stabiliser
	of the element~$x$ resp. the stabiliser of the subset $Y$ (as a subset, not pointwise).
	
	\item Commutators are left normalised:
	$$
	[x,y]=xyx^{-1}y^{-1}.
	$$
	
	\item Upper index stands for the left or right conjugation: 
	$$
	\leftact{g}{h}=ghg^{-1},\quad h^g=g^{-1}hg.
	$$
	
	\item If $X$ is a subset of the group $G$, we denote by $\<X\>$ the subgroup generated by $X$.
	
\end{itemize}

\subsection{Nets of ideals}

A collection of ideals $\sigma=\{\sigma_\alpha\}_{\alpha\in\Phi}$ pf the ring $R$ is called a {\it net of ideals}, if the following conditions are fulfilled.
\begin{enumerate}
	\item If $\alpha, \beta,\alpha+\beta\in\Phi$, then $\sigma_\alpha\sigma_\beta\sub \sigma_{\alpha+\beta}$.
	
	\item if $\alpha\in\Delta$, then $\sigma_\alpha=R$.
\end{enumerate}

When we are considering nets of ideals, we will always assume that $\Delta^\perp=\emp$; in other words, the system $\Phi$ has no roots that are orthogonal to $\Delta$. Assuming that, for every overgroup
$$
E(\Delta,R)\le H\le G(\Phi,R)
$$
we define net of ideals $\lev(H)$, which is called the {\it level} of $H$, as follows:
$$
\lev(H)_\alpha=\{\xi\in R\colon x_\alpha(\xi)\in H\}.
$$
From the paper \cite{VSch}, it follows that this collection of sets is a net of ideals. In addition, in this paper the following observation was made.

\begin{lem}
	\label{Weylorbits} If $\sigma$ is a net of ideals, then the ideal $\sigma_\alpha$ depends not on the root~$\alpha$ itself, but only on its orbit under the action of the Weyl group $W(\Delta)$.
\end{lem}
\subsection{Lie algebras}
We denote by $L(\Phi,\Z) $ the integer span of the Chevalley basis in the complex Lie algebra of type $\Phi$ (see~\cite{Humphreys}). The following notation stands for Chevalley Lie algebra 
$$
L(\Phi,R)=L(\Phi,\Z)\ox_\Z R.
$$
This is a Lie algebra over the ring $R$ equipped with an action of the group $G(\Phi,R)$ called the adjoint representation. For elements $g\in G(\Phi,R)$ and $v\in L(\Phi,R)$, we denote this action by $\leftact{g}{v}$.

Note that the algebra $L(\Phi,R)$, in general, is not isomorphic to the tangent Lie algebra of the algebraic group $G(\Phi,R)$ (see, for example, \cite{RoozemondDiss}). However, if the group is simply connected, then these algebras are canonically isomorphic.

\smallskip

We denote by $e_\alpha$, $\alpha\in\Phi$ and $h_i$, $i=1,\ldots,\rk\Phi$ the Chevalley basis of the Lie algebra $L(\Phi,R)$. By $D$ we denote its toric subalgebra generated by~$h_i$.

For every element $v\in L(\Phi,R)$, we denote by $v^\alpha$ and $v^i$ its coefficient in Chevalley basis. 

Lie bracket is denoted by $[\,\cdot\,,\,\cdot\,]$. The same notation stands for the group commutator, but it should be clear from the context whether the calculation are in a group or in a algebra. 

For every net of ideals $\sigma$, we set
$$
L(\sigma)=D\oplus\bigoplus_{\alpha\in\Phi}\sigma_\alpha e_\alpha\le L(\Phi,R).
$$

The following lemma follows immediately from the definition of a net.

\begin{lem}
	\label{lie} 
	The submodule $L(\sigma)$ is a Lie subalgebra in $L(\Phi,R)$.
\end{lem}

\medskip

Now, we set
$$
L'(\sigma)=\<\{\xi e_\alpha\colon \alpha\in\Phi,\; \xi\in\sigma_\alpha\}\>\le L(\sigma),
$$
where the angle brackets mean generation as a Lie subalgebra. We mention two observations.

\begin{lem}
	\label{L'} Let $\sigma$ be a net of ideals. Then
	\begin{enumerate}
		\item $[L(\sigma),L(\sigma)]\le L'(\sigma);$
		\item $L(\sigma)=\{v\in L(\Phi,R)\colon [v,L'(\sigma)]\le L(\sigma) \}$.
	\end{enumerate}
\end{lem}
\begin{proof}
	The first statement is obvious. In the second statement the inclusion of the left-hand side into the right-hand side is also obvious, we prove the inverse inclusion. Let $v$ be an element from the right-hand side, and $\gamma\in\Phi$. We should prove that $v^\gamma$ is in $\sigma_\gamma$. If $\gamma\in \Delta$, then there is nothing to prove; assume $\gamma\notin\Delta$. Since $\Delta^\perp=\emp$, it follows that there is $\alpha\in\Delta$ such that $(\alpha,\gamma)=-1$, i.e. $\alpha+\gamma\in\Phi$. Then 
	$
	[v,e_\alpha]\in [v,L'(\sigma)]\le L(\sigma).
	$  
	Hence $v^\gamma=\pm[v,e_\alpha]^{\alpha+\gamma}\in\sigma_{\alpha+\gamma}$. However, by Lemma \ref{Weylorbits} we have $\sigma_{\alpha+\gamma}=\sigma_\gamma$.
\end{proof}

\subsection{Net subgroups}

For every net of ideals $\sigma$, we set
\begin{align*}
	E(\Phi,\Delta,R,\sigma)&=\<x_\alpha(\xi)\colon \alpha\in\Phi,\; \xi\in\sigma_\alpha\>\le G(\Phi,R),\\
	S(\Phi,\Delta,R,\sigma)&=\Stab_{G(\Phi,R)}(L(\sigma)).
\end{align*}

When the rest of parameters are clear from the context, we will denote this subgroups simply by $E(\sigma)$ and $S(\sigma)$.

These two subgroups will be the smallest and the biggest subgroup in a sandwich. Therefore, we reformulate the conjectures from \cite{VSch}, replacing nor\-ma\-li\-ser by the subgroup $S(\sigma)$.

Now we make the following simple observation.

\begin{lem}
	\label{Ssigmadesrciption} Let $\sigma$ be a net of ideals. Then
	$$
	S(\sigma)=\big\{g\in G(\Phi,R)\colon \leftact{g}{L'(\sigma)}\le L(\sigma)\text{ and }\leftact{g^{-1}}{L'(\sigma)}\le L(\sigma)\big\}.
	$$
\end{lem}
\begin{proof}
	The inclusion of the left-hand side into the right-hand side is obvious, we prove the inverse inclusion. Let $g$ be an element from the right-hand side. By Lemma \ref{L'}, to check that $\leftact{g}{L(\sigma)}\le L(\sigma)$, it suffices to prove the inclusion
	$$
	[\leftact{g}{L(\sigma)},L'(\sigma)]\le L(\sigma).
	$$ 
	We have
	$$
	[\leftact{g}{L(\sigma)},L'(\sigma)]=\leftact{g}{[L(\sigma),\leftact{g^{-1}}{L'(\sigma)}]}\le \leftact{g}{[L(\sigma),L(\sigma)]}\le \leftact{g}{L'(\sigma)}\le L(\sigma).
	$$
	Similarly we have $\leftact{g^{-1}}{L(\sigma)}\le L(\sigma)$. Thus $g\in S(\sigma)$.
\end{proof}

\section{The statement of the main result} 

\subsection{Pseudo-standard overgroups}

Consider an overgroup 
$$
E(\Delta,R)\le H\le G(\Phi,R).
$$
We say that the overgroup $H$ is {\it pseudo-standard} if
$$
H\le S(\lev(H)).
$$

For given $\Phi$, $\Delta$ and $R$, we say that {\it pseudo-standard description} of overgroups is available if all the overgroups $E(\Delta,R)\le H\le G(\Phi,R)$ are pseudo-standard. 

It will be proved later that $\lev(S(\sigma))=\sigma$ for every net of ideals $\sigma$. Hence every net of ideals can be a level of an overgroup. It follows that pseudo-standard description can be reformulated in the following way: for any overgroup
$$
E(\Delta,R)\le H\le G(\Phi,R)
$$
there exists a unique net of ideals $\sigma$ such that
$$
E(\sigma)\le H \le S(\sigma).
$$

We add the prefix ``pseudo'' because the subgroup $E(\sigma)$ may fail to be normal in $S(\sigma)$, i.e. technically such a result is not a sandwich classification.  

\subsection{Combinatorial condition}
\label{CombCondition}

Fix a pair $(\Phi,\Delta)$. Let $\alpha_1,\alpha_2\in \Delta$ and let $\alpha_1\perp\alpha_2$. We call a pair of roots $\alpha_1,\alpha_2$ {\it admissible} if for every distinct roots $\gamma_1,\gamma_2\in \Sigma_{\alpha_1,\alpha_2}\sm \Delta$ there exists a root $\beta\in \Delta$ such that $\vpi_{\alpha_1,\alpha_2}(\beta)=0$ and $(\beta, \gamma_1)\ne(\beta,\gamma_2)$.

Now we formulate a condition on the pair $(\Phi,\Delta)$:
\begin{equation}\tag{$*$}
	\begin{split}
		&\textit{for any root }\gamma\!\in\!\Phi\sm\Delta\textit{ there exists an admissible pair}
		\\[-1mm]
		&\textit{of orthogonal roots }\alpha_1,\alpha_2\!\in\!\Delta\textit{ such that } (\alpha_1,\gamma)\!=\!(\alpha_2,\gamma)\!=\!-1.
	\end{split}
\end{equation}

In the rest of the paper we always assume that this condition is fulfilled. Note that the condition $\Delta^\perp=\emp$, which we require before, follows from the condition $(*)$.

In Section \ref{examples} we provide many examples of subsystem that satisfy $(*)$.

\subsection{Main theorem}

The main result of the present paper is the following theorem.

\begin{thm}\label{sandwich} Let $R$ be a commutative ring. Let $\Phi$ be an irreducible simply laced root system, and let $\Delta$ be its subsystem that satisfies $(*)$. Then the following statements hold.
	
	\begin{enumerate}
		\item If the ring $R$ has no residue field of two elements, then a pseudo-standard description of overgroups is available.
		
		\item Assume that for fixed $\Phi$ and $\Delta$ a pseudo-standard description of overgroups is available for $R=\F_2$. Then it is available for an arbitrary ring $R$.
	\end{enumerate}
\end{thm}

\section{The intersection with $U'_{\alpha_1,\alpha_2}$}

\begin{prop}
	\label{capwithU} 
	Let $H$ be an overgroup of $E(\Delta,R)$ of level $\sigma$. Let $\alpha_1,\alpha_2\in\Delta$ be an admissible pair of orthogonal roots. Then
	$$
	H\cap U'_{\alpha_1,\alpha_2}\le E(\sigma).
	$$
\end{prop}
\begin{proof}
	Let $g\in H\cap U'_{\alpha_1,\alpha_2}$. Then 
	$
	g=\prod_{i=1}^k x_{\gamma_i}(\xi_i),
	$
	where $\gamma_i$ are distinct roots from $\Sigma_{\alpha_1,\alpha_2}$. We prove that $\xi_i\in\sigma_i$ for all $i$. Assume that for some $g$ the converse is true. Among all such $g$ we chose the one with the minimal $k$. Then $\xi_i\notin\sigma_{\gamma_i}$ for any~$i$ because otherwise, since $U'_{\alpha_1,\alpha_2}$ is Abelian, it follows that we can remove this factor, and make $k$ smaller. In particular, we have $\gamma_i\notin\Delta$.
	
	If $k=1$, then, in view of the above, we obtain a contradiction with the fact that $\lev(H)=\sigma$. Hence $k\ge 2$. By definition of an admissible pair there is a root $\beta\in \Delta$ such that $\vpi_{\alpha_1,\alpha_2}(\beta)=0$, but $(\beta, \gamma_1)\ne(\beta,\gamma_2)$. Either $(\beta, \gamma_1)$ or $(\beta,\gamma_2)$ is not zero. Without loss of generality, we may assume that it is the first one. Next, replacing, if necessary, $\beta$ by $-\beta$, we may assume that $(\beta, \gamma_1)=-1$. Therefore, $\gamma_1+\beta\in\Phi$, but $\gamma_2+\beta\notin\Phi$.
	
	Set
	$$
	g_1=[x_{\beta}(1),g]=\bigg[x_{\beta}(1),\prod_{i=1}^k x_{\gamma_i}(\xi_i)\bigg]= \prod_{i=1}^k [x_{\beta}(1), x_{\gamma_i}(\xi_i)].
	$$
	The last equality follows from the fact that $x_{\beta}(1)\in P_{\alpha_1,\alpha_2}$ (and hence it normalises $U'_{\alpha_1,\alpha_2}$), and the fact that the group $U'_{\alpha_1,\alpha_2}$ is Abelian. 
	
	Each factor in this expansion is equal either to one, or to an elementary root element for a root from $\Sigma_{\alpha_1,\alpha_2}$. The first factor is equal to $x_{\beta+\gamma_1}(\pm\xi_1)$, where $\xi_1\notin\sigma_{\gamma_1}$. Hence by Lemma \ref{Weylorbits} we have $\xi_1\notin\sigma_{\gamma_1+\beta}$. In addition, since the second factor is equal to one, it follows that the number of non-trivial factors is less than $k$, which contradicts the assumption about minimality of $k$.
\end{proof}

\section{Tandems}

We call an element of the set $G(\Phi,R)\times L(\Phi,R)$ a {\it tandem} if it can be writen as $(\leftact{h}{(x_\alpha(\xi))},\leftact{h}{(\xi e_\alpha)})$, where $\alpha\in\Phi$, $\xi\in R$ and $h\in G(\Phi,R)$. 

Actually, it can be shown that the first component of a tandem can be recovered from the second one, but this is inessential for our purposes.

\begin{rem}
	\label{rootchange} 
	For a given root $\beta\in\Phi$ any tandem $(\leftact{h}{(x_\alpha(\xi))},\leftact{h}{(\xi e_\alpha)})$ can be written as $(\leftact{h'}{(x_\beta(\xi'))},\leftact{h'}{(\xi' e_\beta)})$. Moreover, $h'$ can be obtain from $h$ by multiplication by certain element of the extended  Weyl group, and $\xi'=\pm\xi$.
\end{rem}

\begin{lem}
	\label{tandemactpre} The first component of a tandem acts on the element $v\in L(\Phi,R)$ in a following way\textup:
	$$
	\leftact{(\leftact{h}{x_{\alpha}(\xi)})}{v}=v+[\leftact{h}{(\xi e_\alpha)},v]-\xi(\leftact{h^{-1}}{v})^{-\alpha}\cdot\leftact{h}{(\xi e_{\alpha})}.
	$$
\end{lem}
\begin{proof}
	Note that it suffices to prove the formula
	$$
	\leftact{x_{\alpha}(\xi)}{v}=v+[(\xi e_\alpha),v]-v^{-\alpha}\cdot (\xi^2 e_{\alpha}).\eqno(\#)
	$$
	Indeed, if we substitute $\leftact{h^{-1}}{v}$ for $v$ in this formula, and then act by the element $h$ on both sides, then we obtain the required identity.
	
	Moreover, it suffices to prove the equality $(\#)$ in case where $R$ is a ring of polynomials over $\Z$ with $\xi$ and all the coefficients of $v$ being independent variables. Indeed, if we prove this formula for polynomials, then we can prove it for an arbitrary ring~$R$ by applying homomorphism from the polynomial ring to $R$ that sends variables to the corresponding elements. So let $R$ be the ring of polynomials. In this case $2\in R$ is not a zero divisor, and we may write
	$$
	\leftact{x_{\alpha}(\xi)}{v}=v+\big[(\xi e_\alpha),v\big]+{1\over 2}\big[(\xi e_{\alpha}),[(\xi e_{\alpha}),v]\big].
	$$
	It remains to check the relation
	$$
	\big[e_{\alpha},[e_{\alpha},v]\big]=-2 v^{-\alpha}\cdot e_\alpha,
	$$ 
	which follows from the relations in the Chevalley basis.
\end{proof}

\begin{lem}
	\label{tandemact} 
	Let $(g,l)$be a tandem, and $\beta\in\Phi$. Then
	$$
	\leftact{g}{e_\beta}=e_\beta+[l,e_\beta]-l^{-\beta} \cdot l.
	$$
\end{lem}
\begin{proof}
	By remark to the definition of a tandem, we may assume that $(g,l)=(\leftact{h}{(x_\beta(\xi))},\leftact{h}{(\xi e_\beta)})$. By previous Lemma, it suffices to prove that
	$$
	l^{-\beta}=\xi(\leftact{h^{-1}}{e_\beta})^{-\beta}.
	$$
	Moreover, we may assume that $R=\Z[G][\xi]$, and $h=g_{\gen}$ is a generic element. Let $\chi$ be the Killing form on $L(\Phi,R)$. Then
	\begin{align*}
		l^{-\beta}\chi(e_{-\beta},e_\beta)&=\chi(l,e_\beta)=\chi(\leftact{h}{(\xi e_\beta)},e_\beta)=\xi\chi(e_\beta,\leftact{h^{-1}}{e_\beta})
		\\&=\xi(\leftact{h^{-1}}{e_\beta})^{-\beta}\chi(e_\beta,e_{-\beta})= \xi(\leftact{h^{-1}}{e_\beta})^{-\beta}\chi(e_{-\beta},e_\beta).
	\end{align*}
	Since the ring $\Z[G][\xi]$ is a domain, we can cancel $\chi(e_{-\beta},e_\beta)$.
\end{proof}

\begin{prop}
	\label{tandemsinSsigma} 
	Let $(g,l)$ be a tandem\textup, and $\sigma$ be a net of ideals. Then
	$
	g\in S(\sigma)\Longleftrightarrow \linebreak l\in L(\sigma).
	$
\end{prop}
\begin{proof}
	The implication $l\!\in\! L(\sigma)\!\Longrightarrow\! g\!\in\! S(\sigma)$ follows from Lemma~\ref{tandemactpre}.
	
	Conversely, let $g\in S(\sigma)$, and let $\gamma\in \Phi\sm \Delta$. We must check that $l^\gamma\in \sigma_\gamma$. Take $\alpha\in \Delta$ such that $\alpha+\gamma\in\Phi$ (it exists by condition $(*)$). Then $\leftact{g}{e_\alpha}\in L(\sigma)$. Hence by Lemmas \ref{tandemact} and \ref{Weylorbits} we have
	\begin{align}
		l^{-\alpha}l^{\gamma}&=-(\leftact{g}{e_\alpha})^{\gamma}\in \sigma_\gamma,\\
		l^{-\alpha}l^{\gamma+\alpha}\pm l^{\gamma}&=-(\leftact{g}{e_\alpha})^{\gamma+\alpha}\in\sigma_{\gamma+\alpha}= \sigma_\gamma.
	\end{align}
	Furthermore, we have $\leftact{g}{e_{-\alpha}}\in L(\sigma)$; hence
	\begin{align}
		l^\alpha l^\gamma\pm l^{\gamma+\alpha}=-(\leftact{g}{e_{-\alpha}})^{\gamma}\in \sigma_\gamma.
	\end{align}
	
	Multiplying $(3)$ by $l^{-\alpha}$ and adding or subtracting $(2)$, we obtain
	\begin{align}
		l^{-\alpha}l^{\alpha}l^{\gamma}\pm l^\gamma\in\sigma_{\gamma}.
	\end{align}
	Now the congruences $(4)$ and $(1)$ show that $l^\gamma\in\sigma_{\gamma}$.
\end{proof}

\begin{cor}
	\label{levelSsigma} If $\sigma$ is a net of ideals\textup, then $\lev(S(\sigma))=\sigma$. In particular\textup, $E(\sigma)\le S(\sigma)$.
\end{cor} 

\begin{proof}
	Apply Proposition \ref{tandemsinSsigma} to tandems $(x_\alpha(\xi),\xi e_\alpha)$, ${\alpha\!\in\!\Phi}$.
\end{proof}

\begin{cor}\label{transporter} Let $\sigma$ be a net of ideals. Then
	$$
	S(\sigma)=\Tran_{G(\Phi,R)}(E(\sigma),S(\sigma))\cap\left(\Tran_{G(\Phi,R)}(E(\sigma),S(\sigma))\right)^{-1}.
	$$
\end{cor}

\begin{proof}
	The left-hand side is obviously contained in the right-hand side, we prove the inverse inclusion. By Lemma \ref{Ssigmadesrciption} it suffices to check that
	$$
	\leftact{g}{(\xi e_\alpha)}\in L(\sigma)
	$$
	for all $\alpha\in \Phi$, $\xi\in\sigma_\alpha$ and $g$ belonging to the right-hand side (because the right-hand side is invariant under taking the inverse). But this inclusion holds true by Proposition \ref{tandemsinSsigma} because $(\leftact{g}{x_\alpha(\xi)},\leftact{g}{(\xi e_\alpha)})$ is a tandem, and by assumption its first component is in $S(\sigma)$.
\end{proof}

\begin{cor}
	\label{enoughfortandems} Let $H$ be an overgroup of $E(\Delta,R)$ of level
	$\sigma$. Then the following statements are equivalent\textup:
	\begin{enumerate}
		\item[i)] $H$ is pseudo-standard\textup;		
		\item[ii)] for any tandem $(g,l)$ if $g\in H$, then $l\in L(\sigma)$.
	\end{enumerate}
\end{cor}

\begin{proof}
	The second statement follows from the first one by Pro\-po\-si\-tion~\ref{tandemsinSsigma}, we prove the converse implication. Assume that the second statement is true for $H$. By Lemma \ref{Ssigmadesrciption} it suffices to check that
	$$
	\leftact{h}{(\xi e_\alpha)}\in L(\sigma)
	$$
	for any $\alpha\in \Phi$, $\xi\in\sigma_\alpha$ and any $h\in H$. This is true because $(\leftact{h}{x_\alpha(\xi)},\leftact{h}{(\xi e_\alpha)})$ is a tandem and its first component is in $H$.
\end{proof}

\begin{lem}
	\label{tandemsinparabolic} Suppose that $\alpha_1,\alpha_2\in \Phi,$ $\alpha_1\perp \alpha_2,$ and let $(g,l)$ be a tandem such that $l$ is in the Lie algebra $L_{\alpha_1,\alpha_2},$ where
	$$
	L_{\alpha_1,\alpha_2}=D\oplus\bigoplus_{\vpi_{\alpha_1,\alpha_2}(\alpha)\ge 0} R\cdot e_\alpha\le L(\Phi,R).
	$$
	Then $g\in P_{\alpha_1,\alpha_2}$.
\end{lem}

\begin{warn} If the group $G(\Phi,R)$ is not simply connected, then the Lie algebra $L_{\alpha_1,\alpha_2}$ is not the same as the tangent Lie algebra $\Lie(P_{\alpha_1,\alpha_2})$ of the subgroup~$P_{\alpha_1,\alpha_2}$.
\end{warn}

\begin{proof}
	It follows from Lemma \ref{tandemactpre} that $g\in\Stab_{G(\Phi,R)}(L_{\alpha_1,\alpha_2})$.
	Therefore, the required statement is a consequence of the next lemma.
\end{proof}

\begin{lem}
	\label{parabolicasstab} 
	Suppose that $\alpha_1,\alpha_2\in \Phi,$ $\alpha_1\perp \alpha_2$. Then $P_{\alpha_1,\alpha_2}=\Stab_{G(\Phi,R)}(L_{\alpha_1,\alpha_2})$.
\end{lem}
\begin{proof}
	First we consider the case where the group $G(\Phi,R)$ is simply connected. In this case $L_{\alpha_1,\alpha_2}=\Lie(P_{\alpha_1,\alpha_2})$.
	
	Obviously, $P_{\alpha_1,\alpha_2}\sub\Stab_{G(\Phi,R)}(\Lie(P_{\alpha_1,\alpha_2}))$. Denote by
	$$
	\map{\ph}{P_{\alpha_1,\alpha_2}}{\Stab_{G(\Phi,R)}(\Lie(P_{\alpha_1,\alpha_2}))}
	$$
	the inclusion viewed as a morphism of group schemes over $\Z$. Here we should notice that $\Stab_{G(\Phi,R)}(\Lie(P_{\alpha_1,\alpha_2}))$ is defined by equations over $\Z$, namely certain matrix coefficients in the adjoint representation must be equal to zero. We denote the corresponding scheme by $\Stab(\Lie(P_{\alpha_1,\alpha_2}))$. So we must show that the morphism $\ph$ is an isomorphism.
	
	By Theorem 1.6.1 of \cite{WaterhouseAut} and the fact that the parabolic subgroup is smooth (over $\Z$) and, in particular, flat, it suffices to check that for any algebraically closed field $L$ the following statements hold:
	\begin{enumerate}
		\item $\dim((P_{\alpha_1,\alpha_2})_L)\ge \dim_L \Lie ((\Stab(\Lie(P_{\alpha_1,\alpha_2})))_L);$
		
		\item the maps $\ph(L)$ and $\ph(L[\eps]/(\eps^2))$ are injective;
		
		\item all the elements of $\Stab(\Lie(P_{\alpha_1,\alpha_2}))(L)$, that normalise the connected component of identity in $P_{\alpha_1,\alpha_2}$ (i.e. just $P_{\alpha_1,\alpha_2}$), are in $P_{\alpha_1,\alpha_2}$.
	\end{enumerate}
	
	The second item holds true indeed.. The third one follows from the fact that the parabolic subgroup are self-normalised. Since the scheme $P_{\alpha_1,\alpha_2}$ is smooth, it follows that for the first item it suffices to check that
	$$
	\Lie\big((P_{\alpha_1,\alpha_2})_L\big)\ge \Lie \big((\Stab(\Lie(P_{\alpha_1,\alpha_2})))_L\big).
	$$
	The adjoint action of the right-hand side stabilises $\Lie((P_{\alpha_1,\alpha_2})_L)$ (if a group stabilises a subspace, then so do its Lie algebra). Therefore, the inclusion follows from the fact that $\Lie((P_{\alpha_1,\alpha_2})_L)$ is self-normalised.
	
	Now we consider the arbitrary group $G(\Phi,R)$ in the given isogeny class. Let $h\in G(\Phi,R)$. Then there exists a faithfully flat extension $S$ of the ring $R$ and an element $h'\in G_{\SC}(\Phi,S)$ of the simply connected group such that the natural homomorphism maps it to $h$. Therefore,
	\begin{align*}
		h\in P_{\alpha_1,\alpha_2}&\Leftrightarrow h'\in (P_{\alpha_1,\alpha_2})_{\SC}
		\\
		&\Leftrightarrow h'\in \Stab_{G(\Phi,S)}(L_{\alpha_1,\alpha_2})\Leftrightarrow h\in \Stab_{G(\Phi,R)}(L_{\alpha_1,\alpha_2}),
	\end{align*} 
	where $(P_{\alpha_1,\alpha_2})_{\SC}$ is a parabolic subgroup of $G_{\SC}(\Phi,S)$. This concludes the proof.
\end{proof}

\section{Bitandems}
%\rm
We call an element of the set $G(\Phi,R)\times L(\Phi,R)$ a {\it bitandem} if it can be writen as $(\leftact{h}{(x_{\alpha_1}(\xi)x_{\alpha_2}(\zeta))},\leftact{h}{(\xi e_{\alpha_1}+\zeta e_{\alpha_2})}),$ where $\alpha_1,\alpha_2\in\Phi$, $\alpha_1\perp \alpha_2,$ $\xi,\zeta\in R$ and $h\in G(\Phi,R)$.

Similarly, We call a family $(g(t),l(t))\in G(\Phi,R)\times L(\Phi,R),$ $t\in R$ a {\it bitandem with parameter} if
\begin{align*}
	g(t)&=\leftact{h}{((x_{\alpha_1}(t\xi))x_{\alpha_2}(t\zeta))},\\
	l(t)&=\leftact{h}{(t\xi e_{\alpha_1}+t\zeta e_{\alpha_2})}.
\end{align*}

In this case we use the notation $g=g(1)$, $l=l(1)$ for short.

\begin{lem}
	\label{bitandemact} 
	Let  $(g(t),l(t))$ be a bitandem with parameter, and let $v\in L(\Phi,R)$. Then there exists $w\in L(\Phi,R)$ such that for any $t\in R$ we have
	$$
	\leftact{g(t)}{v}=v+t[l,v]+t^2w.
	$$
	In addition, we have $2w=[l,[l,v]]$.
\end{lem}
\begin{proof}
	As in Lemma \ref{tandemactpre}, it suffices to prove this formula for $h=1$ and for such a ring that $2$ is not a zero divizor in it (we may assume that $t$ is a free variable). In this case, the statement is trivial.
\end{proof}

Let $\alpha_1,\alpha_2\in \Phi$, $\alpha_1\perp \alpha_2$, and let $(g,l)$ be a tandem. Set
$$
(g_1(t),l_1(t))=\big(\leftact{g}{(x_{\alpha_1}(tl^{-\alpha_2})x_{\alpha_2}(- tl^{-\alpha_1}))},\leftact{g}{(tl^{-\alpha_2}e_{\alpha_1}- tl^{-\alpha_1}e_{\alpha_2})}\big).
$$

A bitandem with parameter that can be obtain in such a way and the corresponding bitandem $(g_1,l_1)=(g_1(1),l_1(1))$ will be called {\it special} with respect to the pair $\alpha_1$, $\alpha_2$. 

\begin{lem}
	\label{bitandemsinparabolic} 
	Let $(g_1(t),l_1(t))$ be a special with respect to the pair $\alpha_1, \alpha_2$ bitandem with parameter. Then $g_1(t)\in P_{\alpha_1,\alpha_2}$ for all $t\in R$.
\end{lem}
\begin{proof}
	First, it suffices to consider the case where 
	$$
	R=\Z[G][\xi,t],\quad g=\leftact{g_{\gen}}{x_\alpha(\xi)},\quad l=\leftact{g_{\gen}}{(\xi e_\alpha)},
	$$
	and the bitandem with parameter $(g_1(t),l_1(t))$ is obtained from the tandem $(g,l)$ as described above.
	
	Second, note that $l_1\in L_{\alpha_1,\alpha_2}$ (see Lemma \ref{tandemsinparabolic}). Indeed, by Lemma \ref{tandemact} we have
	\begin{align*}
		l_1&=\leftact{g}{(l^{-\alpha_2}e_{\alpha_1}- l^{-\alpha_1}e_{\alpha_2})}
		\\
		&=(l^{-\alpha_2}e_{\alpha_1}- l^{-\alpha_1}e_{\alpha_2})+[l,(l^{-\alpha_2}e_{\alpha_1}- l^{-\alpha_1}e_{\alpha_2})]-l^{-\alpha_2}l^{-\alpha_1}\cdot l+ l^{-\alpha_1}l^{-\alpha_2}\cdot l
		\\
		&=(l^{-\alpha_2}e_{\alpha_1}- l^{-\alpha_1}e_{\alpha_2})+[l,(l^{-\alpha_2}e_{\alpha_1}- l^{-\alpha_1}e_{\alpha_2})].
	\end{align*}
	
	The grading of the root system $\Phi$ by the functional $\vpi_{\alpha_1,\alpha_2}$ induces a grading of the Lie algebra $L(\Phi,R)$. The elements $e_{\alpha_1}$ and $e_{\alpha_2}$ have degree~2. Hence the homogeneous components of both summands in the last expression have non\-ne\-ga\-tive degree, which means exactly that $l_1\in L_{\alpha_1,\alpha_2}$.
	
	Third, $g_1(t)\in \Stab(L_{\alpha_1,\alpha_2})$. Indeed, if $v\in L_{\alpha_1,\alpha_2}$, then by Lemma \ref{bitandemact} we have
	$$
	\leftact{g_1(t)}{v}=v+t[l_1,v]+t^2w.
	$$
	The fist two terms are in $L_{\alpha_1,\alpha_2}$, and about $w$ we know that
	$$
	2w=[l_1,[l_1,v]]\in L_{\alpha_1,\alpha_2}.
	$$
	
	Since $2$ is not a zero divisor in $\Z[G][\xi,t]$, we see that $w\in L_{\alpha_1,\alpha_2}$. Therefore,
	$
	\leftact{g_1(t)}{L_{\alpha_1,\alpha_2}}\le L_{\alpha_1,\alpha_2}.
	$
	Replacing $t$ by $-t$, we obtain the inverse inclusion.
	
	It remains to apply Lemma \ref{parabolicasstab}.
\end{proof}

\section{The case of a field}

\begin{prop}
	\label{field} 
	Let $R=K$ be a field distinct from $\F_2$. Then the pseudo-standard description is available for overgroups of $E(\Delta,K)$.
\end{prop}

\begin{proof}
	
	Let $H$ be an overgroup of $E(\Delta,K)$ of level $\sigma$. By Corollary \ref{enoughfortandems} it suffices to prove that for any tandem $(g,l)$, if $g\in H$, then $l\in L(\sigma)$. 
	
	So let $(g,l)$ be a tandem, and let $g\in H$. Assume that $l\notin L(\sigma)$. This means that there exists a root $\gamma\in\Phi\sm\Delta$ such that $\sigma_\gamma=(0)$, but $l^\gamma\ne 0$. By $(*)$ there exists an admissible pair of orthogonal roots $\alpha_1,\alpha_2\in \Delta$ such that $(\gamma,\alpha_1)=(\gamma,\alpha_2)=-1$.
	
	\subsection*{Case 1} $l^{-\alpha_1}=0$. Consider the tandem $(g_1,l_1)=(\leftact{g}{x_{\alpha_1}(1)},\leftact{g}{e_{\alpha_1}})$.
	Note that by Lemma~\ref{tandemact} we have $l_1=e_{\alpha_1}+[l,e_{\alpha_1}]\in \Lie(P_{\alpha_1,\alpha_2})$. Therefore, by construction and by Lemma~\ref{tandemsinparabolic} we have $g_1\in H\cap P_{\alpha_1,\alpha_2}$.
	
	Next, consider the tandem $(g_2,l_2)=(\leftact{g_1}{x_{\alpha_2}(1)},\leftact{g_1}{e_{\alpha_2}})$. Since the subgroup $U'_{\alpha_1,\alpha_2}$ is normal in $P_{\alpha_1,\alpha_2}$, Proposition~\ref{capwithU} shows that
	$$
	g_2\in H\cap U'_{\alpha_1,\alpha_2}\le E(\sigma)\le S(\sigma).
	$$ 
	Then $l_2\in L(\sigma)$ by Proposition \ref{tandemsinSsigma}.
	
	On the other hand $\gamma+\alpha_1$ and $\gamma+\alpha_1+\alpha_2\in \Phi$; hence by Lemma \ref{Weylorbits} we have 
	$$
	\sigma_{\gamma+\alpha_1+\alpha_2}=\sigma_{\gamma}=0,
	$$
	but
	$$
	l_2^{\gamma+\alpha_1+\alpha_2}=(e_{\alpha_2}+[l_1,e_{\alpha_2}])^{\gamma+\alpha_1+\alpha_2}=\pm l_1^{\gamma+\alpha_1}=\pm(e_{\alpha_1}+[l,e_{\alpha_1}])^{\gamma+\alpha_1}=\pm l^\gamma\ne 0
	$$
	(in the first identity we use Lemma \ref{tandemact}, and the fact that $l_1\in\Lie(P_{\alpha_1,\alpha_2})$ hence $l_1^{-\alpha_2}=0$). This is a contradiction.
	
	\subsection*{Case 2} $l^{-\alpha_1}\ne 0$. We build the following bitandem with parameter special with respect to the pair $\alpha_1,\alpha_2$ 
	$$
	(g_1(t),l_1(t))=\big(\leftact{g}{(x_{\alpha_1}(tl^{-\alpha_2})x_{\alpha_2}(- tl^{-\alpha_1}))},\leftact{g}{(tl^{-\alpha_2}e_{\alpha_1}- tl^{-\alpha_1}e_{\alpha_2})}\big).
	$$
	After that we build the tandem $(g_2,l_2)=(\leftact{g_1(t)}{x_{\alpha_1}(1)},\leftact{g_1(t)}{e_{\alpha_1}})$,
	where $t$ is an element of $K$ to be defined later.
	
	Regardless of $t$, by Lemma \ref{bitandemsinparabolic} we have $g_1(t)\in P_{\alpha_1,\alpha_2}$, and, as in the previous case, we have $g_2\in H\cap U'_{\alpha_1,\alpha_2}\le S(\sigma)$, and $l_2\in L(\sigma)$.
	
	On the other hand, by Lemma \ref{bitandemact} we have
	$$
	l_2^{\gamma+\alpha_1+\alpha_2}=(e_{\alpha_1}+t[l_1,e_{\alpha_1}]+t^2w)^{\gamma+\alpha_1+\alpha_2}=\pm l_1^{\gamma+\alpha_2}t+w^{\gamma+\alpha_1+\alpha_2}t^2.
	$$ 
	This is a polynomial on $t$ of degree at most 2. This polynomial is not zero because from the proof of Lemma \ref{bitandemsinparabolic}, it follows that
	$$
	l_1^{\gamma+\alpha_2}=((l^{-\alpha_2}e_{\alpha_1}- l^{-\alpha_1}e_{\alpha_2})+[l,(l^{-\alpha_2}e_{\alpha_1}- l^{-\alpha_1}e_{\alpha_2})])^{\gamma+\alpha_2}=\pm l^{-\alpha_1}l^{\gamma}\ne 0.
	$$
	Therefore, since $|K|>2$, it follows that we can chose~$t$ to be distinct from the roots of this polynomial, and, as in to the previous case, we obtain a contradiction.
\end{proof}

\section{Reduction lemma}

The following Lemma will help us to reduce our problem for the ring $R$ first to its local quotients, and then to its residue fields. 
\begin{lem}
	\label{reduction}
	Let $H$ be an overgroup of $E(\Delta,R)$, and let $I\unlhd R$ be an ideal. Then
	$
	\rho_I(\lev(H))=\lev(\rho_I(H)),
	$
	where $\rho_I(\sigma)_{\alpha}=\rho_I(\sigma_{\alpha})\unlhd R/I$.
\end{lem}
\begin{proof}
	Let $\lev(H)=\sigma$. For simplicity, we will write 
	$$
	\ovl{X}=\rho_I(X)
	$$
	regardless of the nature of $X$. Fix $\gamma\in \Phi\sm\Delta$. We should prove that $\ovl{\sigma_\gamma}=\lev(\ovl{H})_\gamma$. 
	The inclusion of the left-hand side to the right-hand side is obvious, we prove the inverse inclusion. Let $\xi\in \lev(\ovl{H})_\gamma$. This means that there exists $h\in H$ such that $\ovl{h}=x_\gamma(\xi)$. By $(*)$ there exists an admissible pair of orthogonal roots $-\alpha_1, \alpha_2\in\Delta$ such that $(\gamma,-\alpha_1)=(\gamma,\alpha_2)=-1$. Note that in this case the pair $\alpha_1,\alpha_2$ is also admissible, because it is obtained from the original pair by reflection with respect to $\alpha_1\in\Delta$.
	
	We build the tandem $(g,l)=(\leftact{h}{x_{-\alpha_1}(1)},\leftact{h}{e_{-\alpha_1}})$.
	After that we build the bitandem
	$$
	(g_1,l_1)=\big(\leftact{g}{(x_{\alpha_1}(l^{-\alpha_2})x_{\alpha_2}(- l^{-\alpha_1}))}, \leftact{g}{(l^{-\alpha_2}e_{\alpha_1}- l^{-\alpha_1}e_{\alpha_2})}\big),
	$$
	which is special with respect to the pair $\alpha_1,\alpha_2$.
	
	Finally, we set $g_2=[g_1,x_{\alpha_1}(1)]$.
	
	From Lemma \ref{bitandemsinparabolic}, normality of the subgroup $U'_{\alpha_1,\alpha_2}$ in $P_{\alpha_1,\alpha_2}$ and also Pro\-po\-si\-tion $\ref{capwithU}$ we have
	$g_2\in H\cap U'_{\alpha_1,\alpha_2}\le S(\sigma)$. It is easy to see that this implies $\ovl{g_2}\in S(\ovl{\sigma})$.
	
	On the other hand, making calculation modulo the ideal $I$, we obtain
	\begin{align*}
		\ovl{g}&=\leftact{x_\gamma(\xi)}{x_{-\alpha_1}(1)}=x_{\gamma-\alpha_1}(\pm\xi)x_{-\alpha_1}(1),\\
		\ovl{l}&=\leftact{x_\gamma(\xi)}{e_{-\alpha_1}}=e_{-\alpha_1}\pm \xi e_{\gamma-\alpha_1},
	\end{align*}
	this implies
	\begin{align*}
		\ovl{g_1}&=\leftact{x_{\gamma-\alpha_1}(\pm\xi)x_{-\alpha_1}(1)}{(x_{\alpha_1}(0)x_{\alpha_2}(-1))}=\leftact{x_{\gamma-\alpha_1}(\pm\xi)x_{-\alpha_1}(1)}{x_{\alpha_2}(-1)}
		\\
		&=\leftact{x_{\gamma-\alpha_1}(\pm\xi)}{x_{\alpha_2}(-1)}=x_{\gamma-\alpha_1+\alpha_2}(\pm\xi)x_{\alpha_2}(-1),
	\end{align*}
	and, finally,
	\begin{gather*}
		\ovl{g_2}=[x_{\gamma-\alpha_1+\alpha_2}(\pm\xi)x_{\alpha_2}(-1),x_{\alpha_1}(1)]=[x_{\gamma-\alpha_1+\alpha_2}(\pm\xi),x_{\alpha_1}(1)]=x_{\gamma+\alpha_2}(\pm\xi).
	\end{gather*}
	Therefore, $x_{\gamma+\alpha_2}(\pm\xi)\in S(\ovl{\sigma})$, i.e. by Corollary \ref{levelSsigma} and Lemma \ref{Weylorbits} we have $\xi\in \ovl{\sigma}_{\gamma+\alpha_2}=\ovl{\sigma}_{\gamma}$. This concludes the proof.
\end{proof}

\section{Reduction to local rings with nilpotent\\ maximal ideal}

For Noetherian rings Lemma \ref{reduction} allows to immediately reduce our problem to local rings with nilpotent maximal ideal. All we need is the following ring-theoretic observation.

By $\Max(R)$ we denote the set of maximal ideals of the ring~$R$.

\begin{prop}
	\label{powersofmaximal} 
	Let $R$ be a Noetherian ring\textup, $\sigma$ be a net of ideals. Then
	$$
	S(\sigma)=\bigcap_{\M\in\Max(R),\, k\in \N} \rho^{-1}_{\M^k}\big(S(\rho_{\M^k}(\sigma))\big).
	$$
\end{prop}
\begin{proof}
	The inclusion of the left-hand side to the right-hand side is obvious, we prove the inverse inclusion.	
	For any Noetherian ring $S$, the natural map
	$$
	S\to \prod_{\M\in\Max(S),\, k\in \N} S/\M^k
	$$ 
	is injective (this follows from the injectivity of the map to the product of localisations and the Krull Intersection Theorem). Applying this to the rings $S=R/\sigma_{\alpha}$ and using that $\M^k\sub \rho_{\sigma_{\alpha}}^{-1}(\rho_{\sigma_{\alpha}}(\M)^k)$, we obtain 
	$$
	\sigma_{\alpha}=\bigcap_{\M\in\Max(R),\, k\in \N} \rho^{-1}_{\M^k}(\rho_{\M^k}(\sigma_{\alpha})).
	$$
	Hence
	$$
	L(\sigma)=\bigcap_{\M\in\Max(R),\, k\in \N} \rho^{-1}_{\M^k}\big(L(\rho_{\M^k}(\sigma))\big).
	$$
	The right hand side of the identity to be proved, stabilises each of the subalgebras $\rho^{-1}_{\M^k}(L(\rho_{\M^k}(\sigma)))$; hence it stabilises $L(\sigma)$, i.e. is contained in $S(\sigma)$.
\end{proof}

\begin{cor}
	\label{tolocal} Let $R$ be a Noetherian ring\textup, and let $H$ be an overgroup of $E(\Delta,R)$. Assume that $\rho_{\M^k}(H)$ is pseudo-standard for any $\M\in\Max(R)$ and any $k\in\N$. Then $H$ is pseudo-standard.
\end{cor}
\begin{proof}
	Proposition \ref{powersofmaximal} $+$ Lemma \ref{reduction}.
\end{proof}

\section{Reduction to fields}

In order to reduce our problem to the case of a field, we need several technical lemmas.
\begin{lem}
	\label{nilpotent} 
	Let $H$ be an overgroup of $E(\Delta,R)$ of level $\sigma,$ and let $I\unlhd R$ be a nilpotent ideal. Then
	$$
	(HT(\Phi,R))\cap G(\Phi,R,I)\le T(\Phi,R)E(\sigma).
	$$
\end{lem}
\begin{proof}
	The proof is by induction on $k$, where $k$ is the smallest number such that $I^k=0$. 
	The base of induction is $k=2$. Let $hg\in G(\Phi,R,I)$, where $h\in H$ and $g\in T(\Phi,R)$. Since the ideal $I$ is nilpotent, it follows that for an arbitrary choice of an order on the system $\Phi$ we have
	$$
	G(\Phi,R,I)\le T(\Phi,R,I)U(\Phi,R,I)U^-(\Phi,R,I),
	$$
	where
	\begin{align*}
		T(\Phi,R,I)&=T(\Phi,R)\cap G(\Phi,R,I),\\
		U(\Phi,R,I)&=\<x_{\alpha}(\xi) \colon \alpha\in\Phi^+,\, \xi\in I \>,\\
		U^-(\Phi,R,I)&=\<x_{\alpha}(\xi) \colon \alpha\in\Phi^-,\, \xi\in I \>.
	\end{align*}
	This well-known statement can be found, for example, in \cite[Proposition 2.3]{AbeSuzuki}. Therefore, we have
	$$
	hg=g'\prod_{\beta\in\Phi}x_{\beta}(\xi_{\beta}),
	$$
	where $g'\in T(\Phi,R,I)$ and $\xi_{\beta}\in I$. We show that $\xi_{\gamma}\in \sigma_{\gamma}$ for any $\gamma\in \Phi\sm\Delta$. By $(*)$ there exists an admissible pair of orthogonal roots ${\alpha_1, \alpha_2\in\Delta}$ such that $(\gamma,\alpha_1)=(\gamma,\alpha_2)=-1$. Set $g_1=[x_{\alpha_1}(1),[x_{\alpha_2}(1),hg]]$.
	
	Note that $g_1\in H$. 
	Indeed,
	$$
	[x_{\alpha_2}(1),hg]=[x_{\alpha_2}(1),h]\cdot \leftact{h}{[x_{\alpha_2}(1),g]}=[x_{\alpha_2}(1),h]\cdot \leftact{h}{x_{\alpha_2}(\eps)}\in H.
	$$ 
	
	On the other hand, since $I^2=0$, it follows that the subgroup $G(\Phi,R,I)$ is normal and Abelian. Hence
	\begin{align*}
		g_1&\!=\![x_{\alpha_1}(1),[x_{\alpha_2}(1),g']]\!\cdot\!\prod_{\beta\in\Phi}\!\big[x_{\alpha_1}(1),[x_{\alpha_2}(1),x_{\beta}(\xi_{\beta})]\big]
		\\
		&\!=\!\!\prod_{\beta\in\Phi}\!\big[x_{\alpha_1}(1),[x_{\alpha_2}(1),x_{\beta}(\xi_{\beta})]\big]\!=\!x_{\alpha_1}(\zeta)\!\!\prod_{\{\beta\in\Phi\colon (\beta,\alpha_1)=(\beta,\alpha_2)=-1\}}\!\!x_{\beta+\alpha_1+\alpha_2}(\pm\xi_{\beta}).
	\end{align*}
	
	Here the first factor is $[x_{\alpha_1}(1),[x_{\alpha_2}(1),x_{-\alpha_2}(\xi_{-\alpha_2})]]$, because, using teh relation $\xi_{-\alpha_2}^2=0$, we obtain 
	$$
	[x_{\alpha_2}(1),x_{-\alpha_2}(\xi_{-\alpha_2})]\in x_{\alpha_2}(\pm\xi_{-\alpha_2})T(\Phi,R,I).
	$$
	This can be verified by direct calculation in $\SL(2,R)$.
	
	Therefore, we have $g_1\in H\cap U'_{\alpha_1,\alpha_2}$, and from the proof of Proposition \ref{capwithU} it follows that $\xi_\gamma\in\sigma_{\gamma+\alpha_1+\alpha_2}=\sigma_{\gamma}$ (Lemma \ref{Weylorbits}).
	
	Now we do the induction step from $k-1$ to $k$. By the inductive hypothesis, using Lemma \ref{reduction}, we obtain
	\begin{align*}
		\rho_{I^{k-1}}&(HT(\Phi,R)\cap G(\Phi,R,I))
		\\
		&\le\big(\rho_{I^{k-1}}(H)T(\Phi,R/I^{k-1})\big)\cap G\big(\Phi,R/I^{k-1},\rho_{I^{k-1}}(I)\big)
		\\
		&\le T(\Phi,R/I^{k-1})E(\rho_{ I^{k-1}
		}(\sigma)).
	\end{align*}
	
	Since $E(\sigma)$ maps surjectively onto $E(\rho_{I^{k-1}}(\sigma)
	)$, and $T(\Phi,R)\!$ maps surjectively onto $T(\Phi,R/I^{k-\!1})$ (because $I$ is nilpotent), it follows that
	$$
	HT(\Phi,R)\cap G(\Phi,R,I)\le G(\Phi,R,I^{k-1})T(\Phi,R)E(\sigma).
	$$
	Therefore,
	\begin{align*}
		HT(\Phi,R)\!\cap\! G(\Phi,R,I)&\!\le\!\! G(\Phi,R,I^{k\!-\!1})T(\Phi,R)E(\sigma)\!\cap\! HT(\Phi,R)
		\\
		&\!\le\! \!\big(G(\Phi,R,I^{k\!-\!1})\!\cap\! HT(\Phi,R)E(\sigma)T(\Phi,R)\big)T(\Phi,R)E(\sigma).
	\end{align*}
	Using that $T(\Phi,R)$ normalises $E(\sigma)$, we see that the last expression is equal to
	\begin{align*}
		(G(\Phi,R,I^{k-1})&\cap HE(\sigma)T(\Phi,R))T(\Phi,R)E(\sigma)
		\\
		&=(G(\Phi,R,I^{k-1})\cap HT(\Phi,R))T(\Phi,R)E(\sigma).
	\end{align*}
	Since $(I^{k-1})^2=0$, it follows that we can use the base of induction to conclude that the last expression is contained in
	\begin{equation*}
		T(\Phi,R)E(\sigma)T(\Phi,R)E(\sigma)=T(\Phi,R)E(\sigma).\qedhere
	\end{equation*}
\end{proof}

Let $\Delta'\le\Phi$ be a closed set of roots (i.e if $\alpha,\beta\in\Delta'$ and $\alpha+\beta\in\Phi$, then $\alpha+\beta\in\Delta'$), and let $\Delta\le \Delta'$. For any ring $R$ consider the net of ideals $\sigma_{\Delta'}$ defined as follows
$$
(\sigma_{\Delta'})_\gamma=\begin{cases}
	R,&\gamma\in\Delta',\\
	(0),&\gamma\notin\Delta'.
\end{cases}
$$ 
Set $E(\Delta',R)=E(\sigma_{\Delta'})$. Then the Zariski sheafification of the presheaf $T(\Phi,-)E(\Delta',-)$ is a semidirect product of the extended Chevalley group that corresponds to the subsystem $\Delta'\cap(-\Delta')$ and the unipotent subgroup that corresponds to set $\Delta'\sm(-\Delta')$. We denote this group by $GG(\Delta',R)$. 

Here by extended Chevalley group we mean the group
$$
GG(\Delta'\cap(-\Delta'),R)=T(\Phi,R)G(\Delta'\cap(-\Delta'),R).
$$
We show that the group $GG(\Delta',R)$ is the group of points of a smooth affine group scheme. Note that there exists a closed subtorus $T_0(R)\le T(\Phi,R)$ (possibly trivial), such that
$$
T(\Phi,R)=T(\Delta'\cap(-\Delta'),R)\times T_0(R),
$$
where the factor $T(\Delta'\cap(-\Delta'),R)$ is a toric subgroup of the subsystem subgroup $G(\Delta'\cap(-\Delta'),R)\le G(\Phi,R)$. This follows from the fact that any surjective homomorphism of lattices splits, and from the duality between tori and lattices. Therefore, we have $GG(\Delta'\cap(-\Delta'),R)=G(\Delta'\cap(-\Delta'),R)\leftthreetimes T_0(R)$. In particular, this  is the group of points of a smooth affine group scheme, and the same is true for the group $GG(\Delta',R)$, because its second factor is isomorphic as a scheme to an affine space. 

We also set $\tilde{G}(\Delta',R)=S(\sigma_{\Delta'})$. It is clear that $GG(\Delta',R)\le \tilde{G}(\Delta',R)$.

\begin{lem}
	\label{connectedcomponent}
	Let $R=L$ be an algebraically closed field. Then the subgroups $GG(\Delta',L)$ and $ \tilde{G}(\Delta',L)$ are closed in $G(\Phi,L),$ and the subgroup $GG(\Delta',L)$ is a connected component of identity of the group  $\tilde{G}(\Delta',L)$.
\end{lem}
\begin{proof}
	The statement about being closed in $G(\Phi,L)$ is obvious. It is also obvious that $GG(\Delta',L)$ is connected. It remains to show that $GG(\Delta',L)$ is open in $\tilde{G}(\Delta',L)$. Since $GG(\Delta',L)$ is smooth, by Corollary 5.6 in the book \cite{DemazureGabriel} it follows that it suffices to show that the Lie algebras of these groups coincide. 
	
	Note that $\tilde{G}(\Delta',-)$ is a subscheme (even over $\Z$) in $G(\Phi,-)$. Thus the Lie algebra in question can be identified with the group
	$$
	\tilde{G}(\Delta',L[\eps]/(\eps^2))\cap G(\Phi,L[\eps]/(\eps^2),(\eps)).
	$$
	By Corollary \ref{levelSsigma} we have
	$
	\lev(\tilde{G}(\Delta',L[\eps]/(\eps^2)))=\sigma_{\Delta'}.
	$
	Then by Lemma~\ref{nilpotent} we have
	\begin{align*}
		\tilde{G}(\Delta',&L[\eps]/(\eps^2))\cap G(\Phi,L[\eps]/(\eps^2),(\eps))
		\\
		&\le \big(T(\Phi,L[\eps]/(\eps^2))E(\Delta',L[\eps]/(\eps^2))\big)\cap G(\Phi,L[\eps]/(\eps^2),(\eps))
		\\
		&\le GG(\Delta',L[\eps]/(\eps^2))\cap G(\Phi,L[\eps]/(\eps^2),(\eps)).
	\end{align*}
	This concludes the proof.
\end{proof}

We denote by $N(\Phi,-)\le G(\Phi,-)$ the scheme normaliser of the group subscheme $T(\Phi,-)$. There is a natural homomorphism from the scheme $N(\Phi,-)$ to the Weyl group $W(\Phi)$ viewed as constant scheme. This homomorphism is surjective on point, and its kernel is a subsheme~$T(\Phi,R)$.

We denote by $\ovl{W}(\Phi)$ the image of the group $N(\Phi,\Z)$ under homomorphism $G(\Phi,\Z)\to G(\Phi,R)$ induced by the unique ring homomorphism $\Z\to R$. We have dropped the letter $R$ in this notation because this group almost independent of the ring: the group $\ovl{W}(\Phi)$ is isomorphic to the group $N(\Phi,\Z)$ if $2\ne 0$ on $R$; otherwise, it is isomorphic to the group $W(\Phi)$. 

Next, let $\ovl{W}(\Phi,\Delta')\le \ovl{W}(\Phi)$ be the preimage of the group 
$$
\Stab_{W(\Phi)}(\Delta')\le W(\Phi)
$$
under the natural homomorphism.
\begin{lem}
	\label{decompose} 
	Let $R=K$ be a field. Then $\tilde{G}(\Delta',K)=GG(\Delta',K)\ovl{W}(\Phi,\Delta')$.
\end{lem}
\begin{proof}
	The inclusion of the left-hand side into the right-hand side is obvious, we prove the inverse inclusion. First, let $K=L$ be an algebraically closed field, and let $g\in \tilde{G}(\Delta',L)$. Note that $T(\Phi,L)$ is a maximal torus of the group $GG(\Delta',L)$. By Lemma \ref{connectedcomponent} the subgroup $GG(\Delta',L)$ is normal in $\tilde{G}(\Delta',L)$; hence $\leftact{g}{T(\Phi,L)}\le GG(\Delta',L)$ is another maximal torus. Since the field is algebraically closed, it follows that these tori are conjugate in $GG(\Delta',L)$, i.e. for some $g_1\in GG(\Delta',L)$ we have 
	$$
	T(\Phi,R)=\leftact{g_1g}{T(\Phi,L)}\le GG(\Delta',L).
	$$
	Hence
	\begin{align*}
		g_1g&\in N_{G(\Phi,L)}(T(\Phi,L))\cap \tilde{G}(\Delta',L)=(T(\Phi,L)\ovl{W}(\Phi))\cap \tilde{G}(\Delta',L)
		\\
		&=T(\Phi,L)(\ovl{W}(\Phi)\cap \tilde{G}(\Delta',L))=T(\Phi,L)\ovl{W}(\Phi,\Delta').
	\end{align*}
	
	Consequently, $g\in GG(\Delta',L)T(\Phi,L)\ovl{W}(\Phi,\Delta')=GG(\Delta',L)\ovl{W}(\Phi,\Delta')$.
	Now let $K$ be an arbitrary field, and let $L$ be its algebraically closure. Then
	\begin{align*}
		\tilde{G}(\Delta',K)&=\tilde{G}(\Delta',L)\cap G(\Phi,K)=(GG(\Delta',L)\ovl{W}(\Phi,\Delta'))\cap G(\Phi,K)
		\\
		&=(GG(\Delta',L)\cap G(\Phi,K))\ovl{W}(\Phi,\Delta')=GG(\Delta',K)\ovl{W}(\Phi,\Delta').
	\end{align*}
	This concludes the proof.
\end{proof}

\begin{lem}
	\label{normalforfields} 
	Let $R=K$ be a field, and let $\sigma$ be a net of ideals. Then the subgroup $E(\sigma)$ is normal in the group $S(\sigma)$. 
\end{lem}
\begin{proof}
	The only nets of ideals in a field are $\sigma=\sigma_{\Delta'}$ for closed sets of roots $\Delta'$. Therefore, the statement follows from Lemma \ref{decompose} and the fact that $GG(\Delta',K)=E(\Delta',K)T(\Phi,K)$.
\end{proof}

\begin{prop}
	\label{fromlocaltofield} 
	Let $R$ be a local ring with nilpotent maximal ideal $\M,$ and let $H$ be an overgroup of $E(\Delta,R)$. If $\rho_{\M}(H)$ is pseudo-standard, then so is $H$.
\end{prop}
\begin{proof}
	For convenience, we will write $\ovl{X}=\rho_{\M}(X)$ regardless of the nature of $X$. Let $\lev(H)=\sigma$. Then by Lemma \ref{reduction} we have $\lev(\ovl{H})=\ovl{\sigma}$. Assume that $\ovl{H}$ is pseudo-standard. Then by Lemma \ref{normalforfields} we have $\leftact{\ovl{H}}{E(\ovl{\sigma})}\le E(\ovl{\sigma})$. Since $E(\sigma)$ maps surjectively onto $E(\ovl{\sigma})$, using Lemma \ref{nilpotent}, we obtain
	\begin{align*}
		\leftact{H}{E(\sigma)}&\le (E(\sigma)G(\Phi,R,\M))\cap H=E(\sigma)(G(\Phi,R,\M)\cap H)
		\\
		&\le E(\sigma)T(\Phi,R)\le S(\sigma).
	\end{align*}
	Then, since $H$ is close with respect to taking inverses, by Corollary~\ref{transporter} it follows that $H\le S(\sigma)$. This concludes the proof.
\end{proof}

\begin{cor}
	\label{tofields} Let $R$ be a Noetherian ring\textup, and let $H$ be an overgroup of $E(\Delta,R)$. Assume that $\rho_{\M}(H)$ is pseudo-standard for any $\M\in\Max(R)$. Then $H$ is pseudo-standard.
\end{cor}
\begin{proof}
	Corollary \ref{tolocal} $+$ Proposition \ref{fromlocaltofield}.
\end{proof}

\begin{cor}
	\label{tandemreduction} Let $R$ be a local ring with nilpotent maximal ideal $\M$. Let $H$ be an overgroup of $E(\Delta,R)$ of level $\sigma.$ If $\rho_{\M}(l)\in L(\rho_\M(\sigma))$ for any tandem $(g,l)$ such that $g\in H,$ then $H$ is pseudo-standard.
\end{cor}
\begin{proof}
	Let $H'\le H$ be a subgroup generated by elements $g\in H$ that are the first components of tandems. It is Clear that $\lev(H')=\sigma$. From assumption and Proposition \ref{tandemsinSsigma} it follows that for any tandem $(g,l)$ such that $g\in H$ we have $\rho_{\M}(g)\in S(\rho_{\M}(\sigma))$. Then by definition of $H'$ we obtain $\rho_{\M}(H')\le S(\rho_{\M}(\sigma))$. Moreover, by Lemma \ref{reduction} we have $\lev(\rho_{\M}(H'))=\rho_{\M}(\sigma)$, i.e. the subgroup $\rho_{\M}(H')$ is pseudo-standard. Then by Proposition \ref{fromlocaltofield} the subgroup $H'$ is pseudo-standard. Hence by Corollary \ref{enoughfortandems}, using Proposition \ref{tandemsinSsigma}, we see that $H$ is pseudo-standard. 
\end{proof}

\section{Finishing the proof of Theorem \ref{sandwich}}

For Noetherian rings, the theorem follows from Corollary \ref{tofields} and Proposition~\ref{field}. For arbitrary rings it can be obtain by passing to the inductive limit by finitely generated subrings. It should be noted here that if the ring $R$ has no residue field of two elements, then it is an industive limit of its finitely generated subrings with the same property. Indeed, in this case the elements that can be expressed as $t^2+t$ generate the unit ideal in $R$ (otherwise, the ideal generated by them is contained in some maximal ideal, and in the corresponding residue field all elements are roots of the equation $t^2+t=0$). Hence there exist $a_i, t_i\in R$ such that
$
\sum_{i=1}^n a_i(t_i^2+t_i)=1.
$
Then all finitely generated subrings that contain all the $a_i$ and $t_i$ have no residue field of two elements, and~$R$ is the iductive limit of these subrings. 

\section{Examples of sufficiently large subsystems}
\label{examples}

\begin{lem}
	\label{nA1} 
	Let $\Phi$ be an irreducible simply lased root system, and let $\rk\Phi=n$. Let $\Delta\le\Phi$ be a subsystem that contains a subsystem of type $nA_1$. Then the pair $(\Phi,\Delta)$ satisfy the condition $(*)$.
\end{lem}
\begin{proof}
	Firstly, note that any pair of orthogonal roots from the above subsystem $nA_1$ is admissible. Indeed, let $\alpha_1,\alpha_2$ be such a pair, and let $\gamma_1,\gamma_2\in\Sigma_{\alpha_1,\alpha_2}\sm\Delta$ be distinct roots. Then $\beta$ can be taken to be one of the roots from the subsystem $nA_1$ orthogonal to  $\alpha_1$ and $\alpha_2$. Indeed, if $\gamma_1-\gamma_2$ is orthogonal to all such roots, then since $\gamma_1,\gamma_2\in\Sigma_{\alpha_1,\alpha_2}$ (hence $\gamma_1-\gamma_2\perp \alpha_1,\alpha_2$), it follows that $\gamma_1-\gamma_2$ is orthogonal to the whole subsystem $nA_1$, i.e. $\gamma_1=\gamma_2$.
	
	Now let $\gamma\in\Phi\sm\Delta$. Since $\gamma$ can not be orthogonal to the whole subsystem $nA_1$, it follows that there exists $\alpha_1\in nA_1$ such that $(\gamma,\alpha_1)=-1$. Assume that we can not find second such a root $\alpha_2\in nA_1$ orthogonal to $\alpha_1$. This means that $\gamma$ is orthogonal to the orthogonal complement to the root $\alpha_1$. Hence $\gamma=\pm\alpha_1\in\Delta$, which contradicts the assumption.
\end{proof}

\vspace{-1mm}
The next Proposition shows that Theorem \ref{sandwich} solves most of the problems formulated in the paper \cite{VSch}.

We denote $i$-th simple root by $\eps_i$, and the maximal root by $\delta$. The enumeration of simple roots is the same as in \cite{Bourbaki4-6}.
\vspace{-1mm}
\begin{prop}
	For the following embeddings of root systems, condition $(*)$ is fulfilled {\rm (}here in parenthesis we point out the simple roots of the subsystem $\Delta$ if we need it in the proof{\rm )}.
	\renewcommand{\theenumi}{\alph{enumi}}
	\begin{enumerate}
		\item $A_7\le E_7\quad (-\delta,\eps_1,\eps_3,\ldots,\eps_7)$.
		
		\item $A_5+A_1\le E_6\quad (-\delta,\eps_1,\eps_3,\ldots,\eps_6)$.
		
		\item $A_8\le E_8\quad (\eps_1,\eps_3,\ldots,\eps_8,-\delta)$.
		
		\item $D_5\le E_6\quad (\eps_1,\ldots,\eps_5)$.
		
		\item $E_6\le E_7\quad (\eps_1,\ldots,\eps_6)$.
		
		\item $A_5+A_2\le E_7\quad (-\delta,\eps_1,\eps_2,\eps_4,\ldots,\eps_7)$.
		
		\item $2A_3+A_1\le E_7\quad (-\delta,\eps_1,\eps_2,\eps_3,\eps_5,\eps_6,\eps_7)$.
		
		\item $A_1+A_7\le E_8\quad (-\delta,\eps_1,\eps_2,\eps_4,\ldots,\eps_8)$.
		
		\item $D_5+A_3\le E_8\quad (\eps_1,\ldots,\eps_5,\eps_7,\eps_8,-\delta)$.
		
		\item $2A_4\le E_8\quad (\eps_1,\ldots,\eps_4,\eps_6,\eps_7,\eps_8,-\delta)$.
		
		\item $4A_2\le E_8$.
		
		\item $3A_2\le E_6$.
		
		\item $E_6+A_2\le E_8$.
		
		\item $D_6+A_1\le E_7$.
		
		\item $D_8\le E_8$.
		
		\item $E_7+A_1\le E_8$.
		
		\item $D_4+3A_1\le E_7$.
		
		\item $D_6+2A_1\le E_8$.
		
		\item $2D_4\le E_8$.
		
		\item $7A_1\le E_7$.
		
		\item $2mA_1\le D_{2m}$.
		
		\item $8A_1\le E_8$.
	\end{enumerate}
	\renewcommand{\theenumi}{\arabic{enumi}}
\end{prop}
\begin{proof}
	%\renewcommand{\theenumi}{\alph{enumi}}
	%\begin{enumerate}
	(a) In this case the group $W(\Delta)$ acts on the set $\Phi\sm\Delta$ transitively. Hence we may assume that $\gamma=\eps_2$. Then we can set $\alpha_1=\eps_4$ and $\alpha_2=\eps_3+\eps_4+\eps_5$. We verify that this pair is admissible.
	
	We represent $E_7$ as the set of eight-dimensional vectors that can be obtain by permutation of coordinates in the vectors
	$$
	(1,-1,0,0,0,0,0,0,)\quad\text{and}\quad \left({1\over 2},{1\over 2},{1\over 2},{1\over 2},-{1\over 2},-{1\over 2},-{1\over 2},-{1\over 2}\right).
	$$
	The first type of vectors forms our subsystem $A_7$. In this representation we have
	\begin{align*}
		\alpha_1&=(0,0,0,-1,1,0,0,0),\\
		\alpha_2&=(0,0,-1,0,0,1,0,0).
	\end{align*}
	Then the elements of the set $\Sigma_{\alpha_1,\alpha_2}\sm\Delta$ have the form
	$$
	\left(\,\cdot\,,\,\cdot\,,-{1\over 2},-{1\over 2},{1\over 2},{1\over 2},\,\cdot\,,\,\cdot\,\right).
	$$
	If two such vectors $\gamma_1,\gamma_2$ are distinct, then there exist $i,j\in\{1,2,7,8\}$ such that $\gamma_1$ have $\tfrac{1}{2}$ and $-\tfrac{1}{2}$ at these positions respectively, and vice versa for $\gamma_2$. Then $\beta$ can be taken to be the vector that has $1$ in position $i$ and $-1$ in position $j$.
	
	(b) Similarly.
	
	(c) Similarly. However, here the set $\Phi\sm\Delta$ has two $W(\Delta)$-orbits, but one of them consist of roots that are opposite to the roots of another one; hence it suffices to consider one of the orbits. The verification of the fact that the pair is admissible is similar, but a bit longer.
	
	(d) Similarly.
	
	(e) Similarly.
	
	(f) It is easy to see that for any $\gamma\in\Phi\sm\Delta$ there exist $\alpha_1\in A_2$ and $\alpha_2\in A_5$ such that $(\gamma,\alpha_1)=(\gamma,\alpha_2)=-1$. We check that any such a pair is admissible. Since the group $W(\Delta)$ acts transitively on the set of such pairs, we may assume that $\alpha_1=\eps_1$ and $\alpha_2=\eps_2$. We embed $E_7$ into $\R^9$ so that the first three coordinates correspond to the standard realisation of~$A_2$, and the remaining six coordinates correspond to the standard realisation of~$A_5$. Then we have
	\begin{align*}
		\alpha_1&=(0,-1,1,0,0,0,0,0,0),\\
		\alpha_1&=(0,0,0,-1,1,0,0,0,0).
	\end{align*}
	Let $\gamma_1,\gamma_2\in\Sigma_{\alpha_1,\alpha_2}\sm\Delta$ be a distinct roots. The set $\Phi\sm\Delta$ has two $W(\Delta)$-orbits. Depending on the orbit, the first three coordinates of the root from $\Phi\sm\Delta$ are either permutation of $(\tfrac{1}{3},\tfrac{1}{3},-\tfrac{2}{3})$, or permutation of $(-\tfrac{1}{3},-\tfrac{1}{3},\tfrac{2}{3})$, and the remaining six coordinates are either permutation of $(\tfrac{2}{3},\tfrac{2}{3},-\tfrac{1}{3},-\tfrac{1}{3},-\tfrac{1}{3},-\tfrac{1}{3})$, or permutation of $(-\tfrac{2}{3},-\tfrac{2}{3},\tfrac{1}{3},\tfrac{1}{3},\tfrac{1}{3},\tfrac{1}{3})$. Assume that we can not take $\beta$ to be the root from subsystem~$A_3$ supported in the last four coordinates, i.e. $\gamma_1-\gamma_2$ is orthogonal to all such roots. It is easy to see that this is posible only if the last four coordinates of $\gamma_1$ and $\gamma_2$ coincide (here we use that $(\gamma_1,\alpha_2)=-1$). In particular, this means that $\gamma_1$ and $\gamma_2$ are in the same $W(\Delta)$-orbit. For a given orbit, the condition $\gamma_1,\gamma_2\in\Sigma_{\alpha_1,\alpha_2}$ defines uniquely the coordinates from the second to the fourth. Then the first coordinate is uniquely defined as well, i.e. $\gamma_1=\gamma_2$. 
	
	(g) Similarly, it suffices to show that the pair $\alpha_1\!=\!\eps_3$, $\alpha_2\!=\!\eps_5$ is admissible. This can easily be done case by case, using the corresponding embedding into~$\R^{10}$.
	
	(h) Similarly, considering the corresponding embedding into~$\R^{10}$, it can be shown that any orthogonal pair of roots from $A_7$ is admissible, and this gives sufficiently many admissible pairs.
	
	(i) Similarly, considering the corresponding embedding into~$\R^9$, it can be shown that any pair $\alpha_1\in D_5$, $\alpha_2\in A_3$ is admissible, and this gives sufficiently many admissible pairs. 
	
	(j) Similarly, considering the corresponding embedding into~$\R^{10}$, it can be shown that if the roots $\alpha_1,\alpha_2\in \Delta$ are in distinct components, then such a pair is admissible, and this gives sufficiently many admissible pairs.  
	
	(k) It suffices to show that if $\alpha_1$ and $\alpha_2$ are from distinct components, then such a pair is admissible. Let $\gamma_1,\gamma_2\in\Sigma_{\alpha_1,\alpha_2}\sm\Delta$ be distinct roots. Assume that $\gamma_1-\gamma_2$ is orthogonal to all the remaining components. Considering all possible positions of $\gamma_i$ with respect ot the remaining components, we see that the projection of the vector $\gamma_1-\gamma_2$ to the plane of a component has the square of the length equal either to $0$, or to $\tfrac{2}{3}$. Hence the vector $\gamma_1-\gamma_2$ has the square of the length not greater than~$\tfrac{4}{3}$. However, the difference of two distinct roots has the square of the length always at least 2. This is a contradiction.
	
	(l) Similarly.
	
	(m) Follows formally from the case $4A_2\le E_8$ because $3A_2\le E_6$.
	%\end{enumerate}
	\renewcommand{\theenumi}{\arabic{enumi}}
	
	All the remaining cases follow from Lemma \ref{nA1}.
\end{proof}

\section{Overgroups of $4A_1$ in $D_4$}

As we already noted, overgroups of subsystem subgroups in orthogonal and symplectic groups were described by Vavilov and Schegolev (see references in \cite{VavMIAN} and \cite{SchegDiss}). However, none of those papers cover the case $2mA_1\le D_{2m}$ because it requires a $2A_1$-proof.

From Lemma \ref{nA1}, it follows that condition $(*)$ is fulfilled in this case. Therefore, Theorem \ref{sandwich} solves this problem provided that the ring has no residue field of two elements.    

Now let $\Phi=D_4$ and $\Delta=4A_1$. Already in this case, a pseudo-standard description is not available over $\F_2$. Nevertheless, a weaker form of sandwich classification can be proven for an arbitrary ring. 

Here the set $\Phi\sm\Delta$ has only one $W(\Delta)$-orbit. Hence by Lemma \ref{Weylorbits} we can identify a net of ideals $\sigma$ with the ideal that correspond to this orbit, which makes things easier.

\begin{prop}
	Let $R$ be an arbitrary commutative ring\textup, and let $\sigma\unlhd R$ be an ideal. Then there exists an overgroup 
	$$
	E(4A_1,R)\le\hat{G}(D_4,4A_1,R,\sigma)=\hat{G}(\sigma)\le G(D_4,R)
	$$
	that is maximal among all the overgroups of level $\sigma$.
\end{prop}

\begin{proof}
	%\begin{enumerate}
	(1) Note that it suffices to prove this proposition for Noetherian rings. Indeed, if for Noetherian rings the proposition is proved, then for an arbitrary ring $R$ we can set:
	$$
	\hat{G}(\Phi,\Delta,R,\sigma)=\bigcup_{S_0\in\FG(R)}\;\bigcap_{\substack{S\in\FG(R) \\ S_0\le S}} \hat{G}(\Phi,\Delta,S,\sigma\cap S)\le G(\Phi,R),
	$$
	where $\FG(S)$ is the set of finitely generated subrings in $R$.
	
	(2) Note that it suffices to prove this proposition for local rings with nilpotent maximal ideal. Indeed, if for such rings the proposition is proved, then for an arbitrary Noetherian ring we can set:
	$$
	\hat{G}(\Phi,\Delta,R,\sigma)=\bigcap_{\M\in\Max R,\, k\in \N} \rho^{-1}_{\M^k}\big(\hat{G}(\Phi,\Delta,R/\M^k,\rho_{\M^k}(\sigma))\big).
	$$
	Here the level of the right-hand side is equal to
	$$
	\bigcap_{\M\in\Max R,\, k\in \N} \rho^{-1}_{\M^k}(\rho_{\M^k}(\sigma))=\sigma.
	$$
	We discuss it in the proof of Proposition \ref{powersofmaximal}. Then by Lemma~\ref{reduction} this subgroup is maximal among all the overgroups of level $\sigma$.
	
	Starting with this moment, we assume that $R$ is a local ring with nilpotent maximal ideal $\M$. We will write $\ovl{X}=\rho_{\M}(X)$ regardless of the nature of $X$.
	
	(3) If \!$R/\M$ is distinct from $\F_2$, then by Theorem \ref{sandwich} we can set ${\hat{G}(\sigma)\!\!=\!\!S(\sigma)}$.
	
	(4) If $\sigma\ne \M$, then we also can set $\hat{G}(\sigma)=S(\sigma)$, i.e. any overgroup of $E(\Delta,R)$ of level $\sigma$ is pseudo-standard. Indeed, for $\sigma=R$ there is nothing to prove, it remains to consider the case $\sigma< \M$. Let $H$ be our overgroup, we verify that the conditions of Corollary \ref{tandemreduction} hold true. Let $(g,l)$ be a tandem, and let $g\in H$. We must verify that $\ovl{l}\in L(\ovl{\sigma})=L(0)$. 
	
	Assume that there is a root $\gamma\in \Phi\sm\Delta$ such that $\ovl{l^{\gamma}}\ne 0$. As we know, there exists an admissible pair of orthogonal roots $\alpha_1,\alpha_2\in\Delta$ such that $(\alpha_1,\gamma)=(\alpha_2,\gamma)=-1$. 
	
	By Lemma \ref{reduction} we have $\lev(\ovl{H})=(0)$. Hence if $\ovl{l}^{-\alpha_1}=0$, then we can argue as in the proof of Proposition \ref{field} to obtain a contradiction (in the case where $\ovl{l}^{-\alpha_1}=0$ we did not use the assumption that the field is distinct from~$\F_2$). Let $\ovl{l}^{-\alpha_1}\ne 0$.
	
	We build a bitandem with parameter special with respect to the pair $\alpha_1\!,\alpha_2$:
	$$
	(g_1(t),l_1(t))=\big(\leftact{g}{(x_{\alpha_1}(tl^{-\alpha_2})x_{\alpha_2}(- tl^{-\alpha_1}))},\leftact{g}{(tl^{-\alpha_2}e_{\alpha_1}- tl^{-\alpha_1}e_{\alpha_2})}\big).
	$$
	Next, we build the following tandem that depend on $t\in R$, 
	$$
	(g_2,l_2)=\big(\leftact{g_1(t)}{x_{\alpha_1}(1)},\leftact{g_1(t)}{e_{\alpha_1}}\big).
	$$
	
	Regardless of $t$, by Lemma~\ref{bitandemsinparabolic} we have $g_1(t)\in P_{\alpha_1,\alpha_2}$, and $g_2\in H\cap U'_{\alpha_1,\alpha_2}\le S(\sigma)$. Then $l_2\in L(\sigma)$. Therefore, ${l_2^{\gamma+\alpha_1+\alpha_2}\in\sigma}$. On the other hand, as in the proof of the Proposition~\ref{field}, we can write an equation of the form $l_2^{\gamma+\alpha_1+\alpha_2}=at+bt^2$, where $a,b\in R$ are independent of~$t$; and from this proof it follows that $\ovl{a}\ne 0$, i.e. $a\in R^*$. Then by Nakayama's lemma the elements $at+bt^2$, where $t$ runs through $\M$, generate the ideal $\M$ (because their images generate $\M/\M^2$). Hence $\sigma=\M$, which contradicts the assumption.
	
	(5) To finish the last case, it suffices to consider the case where $R=\F_2$ and $\sigma=(0)$. Indeed, if $R$ is a local ring with nilpotent maximal ideal $\M$, $R/\M=\F_2$ and $\sigma=\M$, then we can set $$\hat{G}(\Phi,\Delta,R,\sigma)=\rho_{\M}^{-1}\big(\hat{G}(\Phi,\Delta,R/\M,(0))\big).$$
	
	(6) So it remains to show that in $G(D_4,\F_2)$ there is only one maximal subgroup containing $E(4A_1,\F_2)=G(4A_1,\F_2)$. This subgroup is $N_{G(\D_4,\F_2)}([G(4A_1,\F_2),G(4A_1,\F_2)])$.
	
	Suppose that $N$ is a maximal subgroup that contains $G(4A_1,\F_2)$. We observe that $(\Z/3\Z)^4\simeq [G(4A_1,\F_2),G(4A_1,\F_2)]\le N$. In the book \cite{ConwayAtlas} all the conjugacy classes of maximal subgroups in $G(D_4,\F_2)$ were listed. After we reject subgroups that have order not divisible by $3^4$, and also subgroups that have $3$-Sylow subgroup of order $3^4$ that is not isomorphic to $(\Z/3\Z)^4$, only one conjugacy class remains. This class consist of normalisers of subgroups that are isomorphic to $(\Z/3\Z)^4$. We must show that the subgroup $(\Z/3\Z)^4$ normalised by $N$ coincides with $[G(4A_1,\F_2),G(4A_1,\F_2)]$. It suffices to show that $N$ contains only one subgroup isomorphic to $(\Z/3\Z)^4$. Suppose the contrary, then the other subgroup (the one that is not necessaryly normal) is contained in some $3$-Sylow subgroup $S$ of the group $N$, the normal subgroup $(\Z/3\Z)^4$ is also contained in $S$ because it is normal. The order of $S$ is equal to $3^5$, and it is a $3$-Sylow subgroup in $G(D_4,\F_2)$. By \cite{ConwayAtlas} the group $G(D_4,\F_2)$ contains a subgroup that is isomorphic to the alternating group $A_9$. Hence~$S$ contains a subgroup isomorphic to $3$-Sylow subgroup of the group $A_9$. Now it could be obtained by elementary means that $S$ can not contain two distinct subgroups isomorphic to $(\Z/3\Z)^4$. This is a contradiction.%\qedhere
	%\end{enumerate}
\end{proof}

In his next paper, the author plans to suggest a way to enlarge the subgroup $E(\Delta,R)$ so that pseudo-standard description became standard (i.e. the lower bound of a sandwich became a normal subgroup of an upper bound), and to show that for subsystems that have no component of type $A_1$ such an enlargement is not required.

%\bibliographystyle{gost71s}
%\bibliography{russian}

\vspace{-1mm}
\begin{thebibliography}{99}
	%\def\selectlanguageifdefined#1{
	%\expandafter\ifx\csname date#1\endcsname\relax
	%\else\language\csname l@#1\endcsname\fi}
	%\ifx\undefined\url\def\url#1{{\small #1}}\else\fi
	%\ifx\undefined\BibUrl\def\BibUrl#1{\url{#1}}\else\fi
	%\ifx\undefined\BibAnnote\long\def\BibAnnote#1{}\else\fi
	%\ifx\undefined\BibEmph\def\BibEmph#1{\emph{#1}}\else\fi
	
	%1
	\bibitem{BV84}
	%	Боревич~З.~И., Вавилов~Н.~А., {\it Расположение подгрупп в полной линейной
	%	группе над коммутативным кольцом}, Тр. Мат. ин-та АН СССР {\bf 165} (1984), 24--42. 
	Z.~I.~Borevich and N.~A.~Vavilov, {\it Arrangement of subgroups in the general linear group over a commutative ring}, Tr. Mat. Inst. Steklov. {\bf165} (1984), 24--42; English transl., Proc. Steklov Inst. Math. \textbf{165} (1985), 27–46. \MR0752930 	
	
	%2
	\bibitem{Bourbaki4-6}
	%	Бурбаки~Н., {\it Группы и алгебры {Л}и}, Гл.~4--6, Мир, Москва, 1972.
	Bourbaki, N. \textit{Elements de mathematique}. Fc. XXXIV. \textit{Groupes et algebres de Lie}. Ch. IV - VI,  Actualites Sci. Industr., No. 1337, Hermann, Paris 1968. MR0240238	
	%3
	\bibitem{VavSplitOrt}
	%	Вавилов~Н.~А., {\it О подгруппах расщепимых ортогональных групп над
	%	кольцом}, Сиб. мат. ж. {\bf 29} (1988), \No4, 537--547.
	N.~A.~Vavilov \textit{Subgroups of split orthogonal groups over a ring }, Sibirsk. Mat. Zh. {\bf 29} (1988), no.4, 31--43; English transl., Sib. Math. J.\textbf{ 29} (1988), no. 4, 537--547. \MR0969101	
	%4
	\bibitem{VavMIAN}
	%	Вавилов~Н.~А., {\it О подгруппах расщепимых классических групп}, Тр. Мат. ин-та РАН {\bf183} (1990), 29--42.
	N.~A.~Vavilov, \textit{Subgroups of splittable classical groups},  Proc. Steklov Inst. Math. 1991, no. 4, 27–41.  Trudy Mat. Inst. Steklov. 183 (1990), 29--42, 223. \MR1092012	
	%5
	\bibitem{VavGav}
	%	Вавилов~Н.~А., Гаврилович~М.~Р., {\it $A_2$-доказательство структурных
	%	теорем для групп Шевалле типов $E_6$ и $E_7$}, Алгебра и анализ {\bf16} (2004), \No4, 54--87.
	N.~A.~Vavilov and M.~R.~Gavrilovich,  {\it An A2-proof of structure theorems for Chevalley groups of types E6 and E7}, Algebra i Analiz {\bf 19} (2004), no. 4, 54--87; English transl., Petersburg Math. J. \textbf{16} (2005), no. 4, 649--672. \MR2090851 	
	%6
	\bibitem{VavStepSubgroupsGLStability}
	%	Вавилов~Н.~А., Степанов~А.~В., {\it Подгруппы полной линейной группы над
	%	кольцом, удо\-влетворяющим условиям стабильности}, Изв. вузов. Мат. \textbf{1989},   вып. 10, 19--25.
	N.~A.~Vavilov, A.~V.~Stepanov, {\it Subgroups of the general linear group over a ring that satisfies stability conditions},  Izv. Vyssh. Uchebn. Zaved. Mat. 1989, , no. 10, 19--25 Soviet Math. (Iz. VUZ) 33 (1989), no. 10, 23--31 \MR1044472	
	%7
	\bibitem{VavStepSurvey}
	%	Вавилов~Н.~А., Степанов~А.~В., {\it Надгруппы полупростых групп}, Вестн. СамГУ. Естественнонаучн. сер. {\bf2008}, \No3, 51--95.
	N.~A.~Vavilov, A.~V.~Stepanov, {\it Overgroups of semisimple groups}, Vestn. Samar. Gos. Univ. Estestvennonauchn. Ser. 2008, no. 3, 51--95. (Russian) \MR2473730 	
	%8
	\bibitem{VSch}
	%	Вавилов~Н.~А., Щеголев~А.~В., {\it Надгруппы subsystem subgroups в
	%	исключительных группах: уровни}, Зап. науч. семин. ПОМИ {\bf400} (2012), 70--126.
	N.~A.~Vavilov, A.~V.~Schegolev, {\it Overgroups of subsystem subgroups in exceptional groups: levels },  Zap. Nauchn. Sem. S.-Peterburg. Otdel. Mat. Inst. Steklov. (POMI)  {\bf 400} (2012), 70--126; English transl., J. Math. Sci. (N.Y.) \textbf{192} (2013), no. 2, 164--195. \MR3029566 	
	%9
	\bibitem{GvozLevi}
	%	Гвоздевский~П.~Б., {\it Надгруппы подгрупп Леви}.~I. {\it Случай абелева
	%	унипотентного радикала}, Алгебра и анализ {\bf31} (2019), \No6, 79--121.
	P.~B.~Gvozdevsky {\it Overgroups of Levi subgroups I. The case of an abelian unipotent radical} Algebra i Analiz \textbf{31} (2019), no. 6, 79--121. \MR4039348 	
	%10
	\bibitem{GolubchikSubgroups}
	%	Голубчик~И.~З., {\it О подгруппах полной линейной группы ${GL}_n({R})$ над
	%	ассоциативным кольцом ${R}$}, Успехи мат. наук {\bf39} (1984), \No1, 125--126.
	I.~Z.~Golubchik, {\it Subgroups of the general linear group ${\rm GL}\sb{n}(R)$ over an associative ring~$R$},  Uspekhi Mat. Nauk \textbf{ 39} (1984), no. 1, 125--126.  Russian Math. Surveys \textbf{39}  (1984), no. 1, 157–158MR0733962	
	%11
	\bibitem{KoibaevBlockdiag}
	%	Койбаев~В.~А., {\it О подгруппах полной линейной группы, содержащих группу
	%	элементарных клеточно-диагональных матриц}, Вестн. Ленингр. ун-та. Сер. Мат. Мех. Астроном.  \textbf{1982}, вып. 13, 33--40.
	V.~A.~Koibaev, {\it On subgroups of the full linear groups containing a group of elementary block diagonal matrices},  Vestnik Leningrad. Univ. Mat. Mekh. Astronom. 1982, 33--40. Vestn. Leningr. Univ., Math. 15, 169-177 (1983)\MR0672594	
	%12
	\bibitem{LuzF4E6}
	%	Лузгарев~А.~Ю., {\it Описание надгрупп ${F}_4$ в ${E}_6$ над коммутативным
	%	кольцом}, Алгебра и анализ {\bf20} (2008), \No6, 148--185.
	A.~Y.~Luzgarev, {\it Description of the overgroups $F_4$ in $E_6$ over a commutative ring}, Algebra i Analiz 20 (2008), no. 6, 148--185 St. Petersburg Math. J. 20 (2009), no. 6, 955--981 \MR2530897	
	%13
	\bibitem{StepOLD}
	%	Степанов~А.~В., {\it Описание подгрупп полной линейной группы над кольцом
	%	при помощи условий стабильности}, Кольца и линейные группы, Кубанский гос. ун-т,  Краснодар, 1988,  с. 82--91.
	A.~V.~Stepanov, {\it Description of subgroups of the general linear group over a ring by means of the stability conditions}.  Rings and linear groups , 82--91, Kuban. Gos. Univ., Krasnodar, 1988. (Russian) MR1206033	
	%14
	\bibitem{StepDiss}
	%	Степанов~А.~В., {\it Структурная теория и подгруппы групп {Ш}евалле над
	%	кольцами}, Доктор. дисс., Санкт-Петребург, 2014.
	A.~V.~Stepanov, {\it Structural theory and subgroups of Chevalley groups over rings},  Doktor. Diss., S.-Peterburg, 2014.	
	%15
	\bibitem{Humphreys}
	%	Хамфрис~Дж., {\it Введение в теорию алгебр Ли и их представлений}, МЦНМО, М., 2003.
	J.~E.~Humphreys, \textit{Algebraic groups and modular Lie algebras}, Mem.  Amer. Math. Soc., vol. 71, Amer. Math. Soc., Providence, R.I. 1967. \MR0217075	
	%16
	\bibitem{SchegMainResults}
	%	Щеголев~А.~В., {\it Надгруппы блочно-диагональных подгрупп гиперболической
	%	унитарной группы над квази-конечным кольцом}: {\it Основные результаты},
	%	Зап. науч. семин. ПОМИ {\bf443} (2016), 222--233.
	A.~V.~Schegolev, \textit{Overgroups of an elementary block-diagonal subgroup of the classical symplectic group over an arbitrary commutative ring}, Zap. Nauchn. Sem. S.-Peterburg. Otdel. Mat. Inst. Steklov. (POMI)  {\bf 443} (2016), 222--233; English transl., J. Math. Sci. (N.Y.) \textbf{222} (2017), no. 4, 516--523. \MR3507773	
	%17
	\bibitem{SchegSymplectic}
	%	Щеголев~А.~В., {\it Надгруппы элементарной блочно-диагональной подгруппы
	%	классической симплектической группы над произвольным коммутативным кольцом}, 
	%	Алгебра и анализ {\bf30} (2018), \No6, 147--199.
	A.~V.~Schegolev, \textit{Overgroups of an elementary block-diagonal subgroup of the classical symplectic group over an arbitrary commutative ring}, Algebra i Analiz \textbf{30} (2018), no. 6, 147--199; English transl., St. Petersburg Math. J.\textbf{ 30} (2019), no. 6, 1007--1041.
	%\MR3882542	
	%18
	\bibitem{AbeSuzuki}
	Abe~E., Suzuki~K., {\it On normal subgroups of Chevalley groups over
		commutative rings}, Tohoku Math. J. {\bf28} (1976), no. 2, 185--198.
	%\MR0439947	
	%19
	\bibitem{Aschbacher84}
	Aschbacher~M., {\it On the maximal subgroups of the finite classical
		groups}, Invent. Math. {\bf76} (1984), no.3, 469--514.
	%\MR0746539	
	%20
	\bibitem{ConwayAtlas}
	Conway~L.~H., Curtis~R.~T., Norton~S.~P. et~al.,{\it 	Atlas of finite groups: maximal subgroups and ordinary characters for simple groups. Maximal subgroups and ordinary characters for simple groups},   Clarendon Press, Oxford Univ. Press, Eynsham, 1985. 
	%\MR0827219 
	
	%21
	\bibitem{DemazureGabriel}
	Demazure~M., Gabriel~P, {\it Introduction to algebraic geometry and
		algebraic groups},  North-Holland Math. Stud., vol. 39, North-Holland Publ. Co., Amsterdam, 1980.
	%\MR0563524	
	%22
	\bibitem{RoozemondDiss}
	Roozemond~D.~A., {\it Algorithms for Lie algebras of algebraic groups},
	Ph.D. thesis, Technische Univ., Eindhoven, 2010.
	
	%23
	\bibitem{SchegDiss}
	Shchegolev~A., {\it Overgroups of elementary block-diagonal subgroups in
		even unitary groups over quasi-finite rings}, Ph.D. thesis, Fak.
	Math. Univ., Bielefeld, 2015.
	
	%24
	\bibitem{StepUniloc}
	Stepanov~A.~V., {\it Structure of {C}hevalley groups over rings via
		universal localization}, J. Algebra {\bf450} (2016), 522--548.
	%\MR3449702
	
	%25
	\bibitem{WaterhouseAut}
	Waterhouse~W.~C., {\it Automorphisms of $\det(x_{ij})$}: {\it the group scheme
		approach}, Adv. in Math. \textbf{65} (1987), no. 2, 171--203. 
	%\MR0900267
	
	%{\bf179} (2003), 99--116.
	
\end{thebibliography}

\end{document}